\documentclass[10.5pt, a4paper]{amsart}
\usepackage{amssymb, amsmath, amscd, bm, amsthm, pdfpages, setspace, hyperref, mathtools, subcaption, graphicx, cite, xfrac, empheq, enumitem}

\def\@Rref#1{\hbox{\rm \ref{#1}}}
\def\Rref#1{\@Rref{#1}}
\theoremstyle{plain}
\newtheorem{theorem}{Theorem}[section]
\newtheorem{proposition}[theorem]{Proposition}

\newtheorem{lemma}[theorem]{Lemma}
\newtheorem{corollary}[theorem]{Corollary}

\theoremstyle{definition}
\newtheorem{definition}{Definition}[section]
\newtheorem{example}[definition]{Example}
\newtheorem{remark}[definition]{Remark}

\begin{document}

\title[Characterization of Decay Rates]{Characterization of Decay Rates for Discrete Operator Semigroups}

\thispagestyle{plain}

\author{Masashi Wakaiki}
\address{Graduate School of System Informatics, Kobe University, Nada, Kobe, Hyogo 657-8501, Japan}
 \email{wakaiki@ruby.kobe-u.ac.jp}

\begin{abstract}
Let $T$ be a power-bounded linear operator on a Hilbert space $X$,
and let $S$ be a bounded linear  operator from another Hilbert space $Y$ to $X$. We investigate 
the non-exponential
rate of decay of $\|T^nS\|$ as $n \to \infty$.
First, when $X = Y$ and
$S$ commutes with $T$, 
we characterize the decay rate of $\|T^nS\|$
in terms of the growth rate of $\|(\lambda I - T)^{-k}S\|$ as 
$|\lambda| \downarrow 1$ for some 
$k \in \mathbb{N}$.
Next, we provide another characterization by means of an 
integral estimate of $\|(\lambda I - T)^{-k}S\|$.
The second characterization is then applied to asymptotic estimates for
perturbed discrete operator semigroups.
Finally, we present some results on
the relation  between the decay rate of $\|T^nS\|$
and the boundedness of the sum $\sum_{n=1}^{\infty} 
f(n)\|T^nSy\|^p$ for all $y \in Y$ in the Banach space setting, where $f \colon \mathbb{N} \to(0,\infty)$ and $p \geq 1$.
\end{abstract}

\subjclass[2020]{Primary 47A10; Secondary 47D03 $\cdot$ 47A30}
\keywords{Asymptotics, 
	Discrete operator semigroup,  Rate of decay, 
	Resolvent condition} 

\maketitle

\section{Introduction}
Let $X$ and $Y$ be complex Banach spaces.
We denote by $\mathcal{L}(Y,X)$ the space of 
bounded linear  operators from $Y$ to $X$, and
set $\mathcal{L}(X) \coloneqq \mathcal{L}(X,X)$.
Let $T \in \mathcal{L}(X)$ be such that 
the spectral radius $r(T)$ of $T$ is at most $1$, and let $S \in \mathcal{L}(Y,X)$.
In this paper, we investigate
the asymptotic behavior of $\|T^nS\|$ as $n \to \infty$ and
relate it to the growth of
$\|(\lambda I - T)^{-k}S\|$ as $|\lambda| \downarrow 1$
for  $k \in \mathbb{N}$.

To motivate our problem,
we consider the following sampled-data system with
sampling period $\tau >0$:
\begin{equation}
	\label{eq:sampled_data_sys}
	\begin{cases}
		\dot z (t) = Az(t)+Bu(t), & t \geq 0;\qquad z(0) = x \in X, \\
		u(t) = Fz(n\tau),& n\tau \leq t < (n+1)\tau,\,n\in \mathbb{N}_0,
	\end{cases}
\end{equation}
where $X$ and $U$ are complex Banach spaces, $A$ is the generator
of a $C_0$-semigroup $(e^{At})_{t \geq 0}$ on $X$,
$B \in \mathcal{L}(U,X)$, $F \in \mathcal{L}(X,U)$, and 
$\mathbb{N}_0 \coloneqq \mathbb{N}
\cup \{ 0\}$.
In the context of control systems, $z$ is the state of the plant and is
sampled at time $t=n\tau$ for each $n \in \mathbb{N}_0$. 
The control input $u$ is generated from the sampled state $z(n\tau)$.

Let $z_{\rm d}(n) \coloneqq z(n\tau)$ for $n \in 
\mathbb{N}_0$, and define $T \in \mathcal{L}(X)$ by
\[
T\xi \coloneqq e^{A\tau}\xi + \int^{\tau}_0 e^{At} BF \xi dt,\quad \xi \in X.
\] 
Then the discrete trajectory $z_{\rm d}$ satisfies
\[
z_{\rm d} (n+1) = Tz_{\rm d}(n),\quad 
n \in \mathbb{N}_0;\qquad z_{\rm d}(0) = x.
\]
This discrete-time system can be regarded as a discretized version
of
the following continuous-time system:
\begin{equation}
	\label{eq:cont_time_sys}
	\dot z_{\rm c}(t) = (A+BF)z_{\rm c}(t),\quad t \geq 0; 
	\qquad z_{\rm c}(0) = x.
\end{equation}
Intuitively, the sampled-data system \eqref{eq:sampled_data_sys}
approximates the continuous-time system \eqref{eq:cont_time_sys} when
the sampling period $\tau$ is sufficiently small. 
However,
the asymptotic behavior of the sampled-data system \eqref{eq:sampled_data_sys}
may differ from that of the continuous-time system \eqref{eq:cont_time_sys}  when $X$ is an infinite-dimensional space. In fact, even exponential stability 
is not preserved for a certain system; see 
\cite[Remark~3.3]{Logemann2003} and
\cite{Rebarber2002}.
This motivates us to find 
sufficient conditions for preservation of stability, and 
such conditions were developed for exponential stability in \cite{Logemann2003,Rebarber2006}, for
strong stability in \cite{Wakaiki2021SIAM}, and for polynomial stability in \cite{Wakaiki2023ESAIM}. Here we are interested in 
the decay rate of  the solution $z$ of 
the sampled-data system 
\eqref{eq:sampled_data_sys} with a smooth initial state $x \in  D(A)$.
For simplicity, assume that $A$ is invertible. Then
the domain $D(A)$ of $A$ coincides with the range of $A^{-1}$.
Since there exists a constant $M>0$ such that 
$\|z(t)\| \leq M \|z_{\rm d}(n)\|$ for all $t \in [n \tau,(n+1)\tau )$
and $n \in \mathbb{N}_0$,
uniform rates of decay for the solutions $z$ with initial states $x \in  D(A)$ are obtained by investigating the quantitative behavior of
$\|T^n A^{-1}\|$.

We begin by reviewing the existing literature in the case $S = I$.
It is well known that the following discrete analogue of one implication in
the Hille-Yosida theorem for $C_0$-semigroups holds:
If $T$ is {\em power-bounded}, i.e., satisfies 
$\sup_{n \in \mathbb{N}_0} \|T^n\| < \infty$,
then $r(T) \leq 1$ and there exists a 
constant $M>0$ such that 
\begin{equation}
	\label{eq:Kreiss_cond_intro}
	\|(\lambda I - T)^{-k}\| \leq \frac{M}{(|\lambda|-1)^k}
\end{equation}
for all $k \in \mathbb{N}$ and $|\lambda| > 1$; see, e.g.,
\cite{Lubich1991}.
However, the converse implication is not true; see \cite{McCarthy1971} for a 
counterexample.

We call $T \in \mathcal{L}(X)$  a {\em 
	Kreiss operator}
if $r(T) \leq 1$ and 
there exists $M>0$
such that \eqref{eq:Kreiss_cond_intro} holds for $k=1$.
If $T$ is a Kreiss operator, then $\|T^n\| = O(n)$ as $n \to \infty$, i.e.,
there exist constants $K>0$ and $n_0 \in \mathbb{N}_0$ such that $\|T^n\| \leq Kn$
for all $n \geq n_0$; see \cite{Lubich1991}.
This growth rate $n$ cannot in general be  improved, as shown in \cite{Shields1978}.
If $T \in \mathcal{L}(X)$ satisfies $r(T) \leq 1$ and if there exist
constants $\alpha \geq 0$ and $M_{\alpha} >0$ such that 
\[
\|(\lambda I - T)^{-1}\| \leq \frac{M_\alpha}{(|\lambda|-1)^\alpha},
\]
for all $1 < |\lambda| < 2$,
then $\|T^n\|= O(n^{\alpha})$ as $n \to \infty$.
Conversely, if $\|T^n\|= O(n^{\alpha})$ as $n \to \infty$ for some 
$\alpha \geq 0$, then
\[
\|(\lambda I - T)^{-1}\| \leq \frac{M_{\alpha+1}}{(|\lambda|-1)^{\alpha+1}},
\]
for all $1 < |\lambda| < 2$ and some constant $M_{\alpha+1} > 0$; see \cite{Eisner2007} for these two implications.
In the Hilbert space setting, a Kreiss operator $T$ satisfies $\|T^n\| = O(n/\sqrt{\log n})$
as $n \to \infty$, which was independently proved in \cite{Cohen2020,Bonilla2021}.
For a UMD space $X$, the estimate $\|T^n\| = O(n/(\log n)^{\beta})$ 
was obtained  in 
\cite{Cuny2020}, where $\beta\in(0,1/2]$ is the reciprocal of the minimum of
the finite cotypes of $X$ and its dual.

We call $T \in \mathcal{L}(X)$ a {\em strong
	Kreiss operator} if $r(T) \leq 1$ and 
there exists $M>0$
such that \eqref{eq:Kreiss_cond_intro} holds for all $k\in \mathbb{N}$.
If $T$ is a strong Kreiss operator, then
$\|T^n\| = O(\sqrt{n})$ as $n \to \infty$; see
\cite{McCarthy1971,Lubich1991}. Moreover,
the growth rate $\sqrt{n}$ cannot in general be 
replaced by a better one; see \cite{Lubich1991}.
It was proved in \cite{Cohen2020} that 
a strong Kreiss operator $T$ on a Hilbert space
satisfies $\|T^n\| = O((\log n)^{\beta})$
as $n \to \infty$, where $\beta>0$ is a constant dependent on $T$.
Other estimates were obtained in the case where $X$ 
is the $L^p$-space with $1 < p < \infty$ \cite{Arnold2023Kreiss} and 
in the case where $X$ is a UMD space \cite{Deng2024}.
A survey on Kreiss operators and 
strong Kreiss 
operators can be found in \cite{Cohen2023}.

Next, we review existing studies in the case $S = I-T$.
We call $T \in \mathcal{L}(X)$ a {\em Ritt operator} if 
$r(T) \leq 1$ and there exists a 
constant $M>0$ such that 
\begin{equation}
	\label{eq:Ritt_cond}
	\|(\lambda I - T)^{-1}\| \leq \frac{M}{|\lambda-1|}
\end{equation}
for all $|\lambda| > 1$.
The estimate~\eqref{eq:Ritt_cond}
on the sector 
\[
\{\lambda \in \mathbb{C} \setminus \{1 \}:
|\arg (\lambda-1)| < \pi/2 + \delta
\}
\] 
for some $\delta \in (0,\pi/2)$ implies that 
$T$ is power-bounded and $\|T^n(I-T)\| = O(n^{-1})$
as $n \to \infty$, as shown in \cite{Komatsu1968}.
The converse statement was also true; see \cite{Nevanlinna1997}.
It was later proved independently in 
\cite{Nagy1999, Lyubich1999} that
$T$ is a Ritt operator if and only if
$T$ is power-bounded and $\|T^n(I-T)\| = O(n^{-1})$
as $n \to \infty$.
Since the estimate $\|T^n(I-T)\| = O(n^{-1})$
as $n \to \infty$ does not imply that 
$T$ is power-bounded (see \cite{Kalton2004}),
power-boundedness cannot be omitted from
the above characterization of Ritt operators.
Estimates for $(\lambda-1) (\lambda I - T)^{-1}$
were also studied in connection with maximum regularity
problems for discrete-time evolution equations; see, e.g.,
\cite{Blunck2001,Blunck2001JFA,Kalton2008,Kemmochi2016}.
In \cite{El-Fallah2002},
a generalization of Ritt operators and Kreiss operators
was introduced and applied to derive
norm estimates
for $T^n$.
See 
\cite{Cohen2023}
for further information on Ritt operators.

In \cite{Nevanlinna2001}, it was shown that
if $T$ is a Kreiss opeator and if there exist constants $M>0$
and $\alpha \in (0,1]$ such that  
\begin{equation}
	\label{eq:gen_Ritt}
	\|(\lambda I- T)^{-1}\| 
	\leq \frac{M}{|\lambda-1|^{1/\alpha}}
\end{equation}
for all $1 < |\lambda| < 2$,
then 
\begin{equation}
	\label{eq:gen_Ritt_decay}
	\|T^n(I-T)\| = O\left(  \frac{1}{n^{2\alpha - 1}} \right)
\end{equation}
as $n \to \infty$.
The following characterization for the specific decay rate $n^{-1/2}$ was obtained in \cite{Dungey2008}:
$T$ is power-bounded and satisfies 
$\|T^n(I-T)\| = O(n^{-1/2})$ as $n \to \infty$ if and only if
$T = cI + (1-c)S$ for some constant $c \in (0,1)$ and
power-bounded operator $S$.
In
\cite{Dungey2009},
the decay estimate $\|T^n(I-T)\| = O(n^{-\alpha})$ 
as $n \to \infty$
in the case $1/2 <  \alpha < 1$
was also characterized in terms of 
norm estimates for all
powers of the resolvent.

Given a power-bounded operator
$T$ and a constant $\alpha \in (0,1]$,
the estimate \eqref{eq:gen_Ritt} holds
for all $1 < |\lambda| < 2$ and some constant $M>0$ if and only if the intersection
of the spectrum of $T$ and the unit circle is contained in $\{1\}$
and 
$\|(e^{i \theta} I- T)^{-1}\| = O(|\theta|^{-1/\alpha})$
as $\theta \to 0$, which
can be obtained from 
Theorem~3.10 and the proof of Lemma~3.9 in 
\cite{Seifert2016}; see also \cite{Paunonen2015}
and \cite{Cohen2016}.
Using this equivalence, 
we see from \cite{Seifert2015,Seifert2016} that 
the estimate \eqref{eq:gen_Ritt_decay} can be improved to
\begin{equation}
	\label{eq:gen_Ritt_decay_S}
	\|T^n(I-T)\| = O\left(
	\left( \frac{\log n}{n} \right)^{\alpha}  \right)
\end{equation}
as $n \to \infty$. 
It was also shown in \cite{Seifert2016} that
the upper bound given in \eqref{eq:gen_Ritt_decay_S}
is sharp for $0 < \alpha < 1/2$ in general.
The growth rate $|\theta|^{-1/\alpha}$ of $\|(e^{i \theta} I- T)^{-1}\|$
as $\theta \to 0$ can be generalized to
a continuous non-increasing function; see \cite{Seifert2015,Seifert2016} for details.
As also shown in \cite{Seifert2016}, when $X$ is a Hilbert space,
the estimate~\eqref{eq:gen_Ritt_decay_S}
can be strengthen to $\|T^n(I-T)\| = O(n^{-\alpha})$
as $n \to \infty$.
This result in the Hilbert space setting
was extended from the polynomial case to the wider class
of functions having reciprocally positive increase in \cite{Ng2020}.
The above-mentioned results in \cite{Seifert2015,Seifert2016,Ng2020} are discrete analogues
of those established 
for $C_0$-semigroups in \cite{Batty2008,Borichev2010,Rozendaal2019}; see also the survey article 
\cite{Chill2020}
for the $C_0$-semigroup results.
For more information 
on estimates for the rate of decay of $\|T^n(I-T)\|$, 
we refer to \cite{Batty2022Survey}, which
also includes a survey 
on the asymptotic behavior of $\|T^nf(T)\|$ for bounded 
holomorphic 
functions $f$ on $\mathbb{D}$.

We characterize the decay rate of $\|T^n S\|$
in terms of the  growth rate of $\|(\lambda I - T)^{-k}S\|$ as $|\lambda|
\downarrow 1$ for some $k \in \mathbb{N}$ when $X=Y$ and $S$ commutes with $T$.
A special case of this characterization is stated as follows.
\begin{theorem}
	\label{thm:intro1}
	Let $X$ be a Hilbert space. Let 
	$T \in \mathcal{L}(X)$ be 
	power-bounded, and
	suppose that $S \in \mathcal{L}(X)$ commutes with $T$.
	Then the following statements are equivalent
	for fixed $\alpha > 0$ and $k \in \mathbb{N}$ satisfying
	$k > \alpha$:
	\begin{enumerate}[label=\upshape(\roman*), leftmargin=*, widest=ii]
		\item $\displaystyle
		\|T^nS\| = O \left(
		\frac{1}{
			n^{\alpha}}
		\right)$ as $n \to \infty$.
		\item 
		There exists a constant $M>0$ such that 
		\[
		\|(\lambda I - T)^{-k} S\| \leq 
		\frac{M}{(|\lambda|-1)^{k-\alpha}}
		\]
		for all $|\lambda| > 1$.
	\end{enumerate}
\end{theorem}
We also present a simple example showing that 
the estimates given in Theorem~\ref{thm:intro1}
cannot in general be improved.
Although 
the special case above is limited to polynomials, the proposed 
approach
can be extended to
a class of functions called $\alpha$-bounded 
regularly varying functions for $\alpha \geq 0$,
which are variants of regularly varying functions;
see Theorem~\ref{sec:main_result} for 
the extended case.
Regularly varying functions are products of
polynomials and slowly varying functions, and 
give natural refinements of polynomial scales.
This class of functions
have been extensively applied in 
various areas of analysis, including
asymptotic estimates for $C_0$-semigroups as discussed in \cite{Batty2016}.

We provide another characterization of the decay rate of $\|T^nS\|$
through an integral estimate for
resolvents.
While 
here we state the polynomial case,
it can also be extended to $\alpha$-bounded regularly varying functions; see Theorem~\ref{thm:int_cond_chara} for 
the extended case.
\begin{theorem}
	\label{thm:intro2}
	Let $X$ be a Hilbert space, and let $Y$ be a Banach space.
	Let 
	$T \in \mathcal{L}(X)$ be power-bounded, and let 
	$S \in \mathcal{L}(Y,X)$.
	Then the following statements are equivalent for fixed $\alpha 	> 0$ and $k \in \mathbb{N}$ satisfying
	$k > \alpha + 1/2$:
	\begin{enumerate}[label=\upshape(\roman*), leftmargin=*, widest=ii]
		\item $\displaystyle
		\|T^nS\| = O \left(
		\frac{1}{
			n^{\alpha}}
		\right)$ as $n \to \infty$.
		\item There exists a constant $M>0$ such that 
		\begin{equation}
			\label{eq:integral_estimate_intro}
			\sup_{1<r<2} (r-1)^{2k-2\alpha - 1}\int_{0}^{2\pi}
			\|R(r e^{i \theta}, T)^k S y\|^2 d\theta \leq M \|y\|^2
		\end{equation}
		for all $y \in Y$.
	\end{enumerate}
\end{theorem}
The integral estimate~\eqref{eq:integral_estimate_intro}
is rooted in the characterization of bounded $C_0$-semigroups
established in \cite{Gomilko1999} and \cite{Shi2000}; see also \cite{Cojuhari2008} and \cite[Theorem~II.1.12]{Eisner2010}
for its discrete analogue.
Theorem~\ref{thm:intro2} can be viewed as 
an extension of the case $k=1$ obtained in \cite{Wakaiki2023ESAIM}, which considered 
the norms $\|T^n Sy\|$ of 
individual orbits rather than the operator norm $\|T^nS\|$ as in statement (i).

We apply
the second characterization
to robustness analysis for semigroup stability.
The robustness problem
we study is stated as follows: Assume that a power-bounded operator $T$
on a Hilbert space satisfies $\|T^nS\| = O(1/f(n))$
as $n \to \infty$ for some $\alpha$-bounded regularly
varying function $f$.
Given a perturbation $D \in \mathcal{L}(X)$,
does the perturbed operator $T+D$ remain power-bounded
and satisfy $\|T^nS\| = O(1/f(n))$
as $n \to \infty$?
For $C_0$-semigroups, several conditions for preservation of
non-exponential stability were developed; see
\cite{Paunonen2011,Paunonen2012SS,Paunonen2013SS,Rastogi2020,Baidiuk2022}  for polynomial stability and \cite{Paunonen2014JDE,Paunonen2015Springer,Rastogi2020} for
strong stability. 
For discrete operator semigroups, robustness of strong stability
was also explored in \cite{Paunonen2015}.
In Corollary~\ref{coro:robustness},
we present a class of perturbations preserving the
power-boundedness of $T$ and 
the decay property of $\|T^nS\|$.

Finally, we examine the relation between
the decay rate of $\|T^nS\| $ and 
the following summability condition
on orbits weighted by a  function $f\colon \mathbb{N} \to (0,\infty)$: There exist constants $C>0$ and $p \geq 1$ such that 
\begin{equation}
	\label{eq:summability_intro}
	\sum_{n=1}^{\infty} f(n) \|T^n Sy\|^p \leq C \|y\|^p
\end{equation}
for all $y \in Y$.
While it is assumed in
Theorems~\ref{thm:intro1} and \ref{thm:intro2} that $X$ is a Hilbert space, 
here we consider the Banach space setting.
It was shown in \cite{Zabczyk1974} that 
when $X= Y$, $S = I$, and $f(n) \equiv 1$,
the condition \eqref{eq:summability_intro}
implies $r(T)  < 1$; see \cite{Przyluski1988} for a different proof.
This result was strengthen to the case of weak orbits 
$|\phi(T^nx)|^p$  in \cite{Weiss1989IJSS}, where $x \in X$ and $\phi$ is a bounded
linear functional on $X$; see also
\cite{Neerven1995,Gluck2015} for further generalizations.
We show in Proposition~\ref{prop:weighted_sum_decay}
that if 
a power-bounded operator $T$ satisfies
the summability condition
\eqref{eq:summability_intro}, then
\[
\|T^nS\| = O\left(
\frac{1}{F(n)^{1/p}}
\right)
\]
as $n \to \infty$, where
\[
F(n) \coloneqq \sum_{m=1}^n f(m),\quad 
n \in \mathbb{N}.
\]
Moreover, we prove 
that the converse statement is true in the case when
$T$ is a multiplication operator on $L^q$-space with
$1 \leq q \leq p < \infty$,  $S = I-T$, and $f(n) = n^{\alpha}$
for some $\alpha \geq 0$; see
Proposition~\ref{prop:multiplication_summability}. This is established by using
intermediate properties between
Kreiss and Ritt operators
investigated in \cite{Mahillo2024}.
The  summability condition \eqref{eq:summability_intro}
with $f(n) \equiv 1$
can be regarded as a discrete counterpart
to the $C_0$-semigroup case discussed in \cite{Wakaiki2024JMAA,Paunonen2025}.

This paper is structured as follows.
In Section~\ref{sec:preliminaries}, we give preliminaries
on Parseval's equality and $\alpha$-bounded
regularly varying functions.
Section~\ref{sec:resolvent_growth} is devoted 
to characterizing the decay rate of $\|T^nS\|$
in terms of the growth rate of $\|(\lambda I - T)^{-k}S\|$
for some $k \in \mathbb{N}$.
In Section~\ref{sec:integral_estimate},
we present an integral estimate for resolvents
to characterize the rate of decay of  $\|T^nS\|$, and in Section~\ref{sec:perturbation}, this result 
is applied to asymptotic estimates
for perturbed discrete operator semigroups.
Finally, in Section~\ref{sec:summability}, 
we connect the decay rate of $\|T^nS\|$ to the summability condition \eqref{eq:summability_intro}.

\paragraph{Notation}
Let $\mathbb{Z}$ and $\mathbb{N}$ denote
the set of integers and
the set of positive integers, respectively. Define 
$\mathbb{N}_0 \coloneqq \mathbb{N} \cup \{0 \}$.
We write 
\begin{align*}
	\mathbb{T} &\coloneqq \{\lambda \in \mathbb{C}:
	|\lambda| = 1 \}, \\
	\mathbb{D} &\coloneqq \{\lambda \in \mathbb{C}:
	|\lambda| < 1 \}, \\
	\mathbb{E} &\coloneqq \{\lambda \in \mathbb{C}:
	|\lambda| > 1 \}.
\end{align*}
The closure of a subset $\Omega $ of $\mathbb{C}$ is 
denoted by $\overline{\Omega}$.
Given functions $f,g \colon [a,\infty)\to (0,\infty)$ for some $a \geq 0$, 
the notation
$f(t) = O(g(t))$ as $t \to \infty$ means that there exist
constants $M>0$ and $t_0 \geq a$ such that 
$f(t) \leq M g(t)$ for all $t \geq t_0$. Similar notation is employed for other types of asymptotic behavior.
We denote by $\Gamma$ the gamma function.

Throughout this paper, we assume that 
all Banach spaces are complex.
Let $X$ and $Y$ be Banach spaces.
We denote by $\mathcal{L}(X,Y)$ 
the space of bounded linear  operators from $X$ to $Y$, and
write $\mathcal{L}(X) \coloneqq \mathcal{L}(X,X)$.
Let $T \in \mathcal{L}(X)$. We denote the spectrum of $T$ by $\sigma(T)$, the resolvent set of $T$ by $\varrho(T)$, and
the spectral radius of $T$ by $r(T)$.
For $\lambda \in \varrho(T)$, we write 
$R(\lambda,T) \coloneqq (\lambda I - T)^{-1}$.

\section{Preliminaries}
\label{sec:preliminaries}
In this section, we first recall
the vector-valued version of
Parseval's equality.
Next, we introduce the notion of $\alpha$-bounded regularly varying functions
and give some properties of this class of functions.
\subsection{Parseval's equality}
First, we recall the well-known resolvent representation of the power $T^n$ 
for a bounded linear operator $T$.
This representation in the case $k=1$ can be obtained from
the Dunford--Riesz functional calculus, and then the case $k \geq 2$
follows by applying  integration by parts to the case $k=1$; see also \cite[p.~2]{McCarthy1971}.
\begin{lemma}
	\label{lem:T_power_rep}
	Let $X$ be a Banach space, and let $T \in \mathcal{L}(X)$.
	Then
	\[
	\binom{n+k-1}{k-1}
	T^n = 
	\frac{r^{n+k}}{2\pi} \int_0^{2\pi} e^{i \theta (n+k)}
	R(r e^{i \theta},T)^k  d\theta
	\]
	for all $n,k \in \mathbb{N}$ and  $r > r(T)$.
\end{lemma}

Let $Y$ be a Hilbert space. If $f \in L^2((0,2\pi);Y )$, then
the following Parseval's equality holds:
\[
\frac{1}{2\pi}
\int^{2\pi}_0 \|f(\theta)\|^2 d\theta = 
\sum_{n=-\infty}^{\infty} \|c_n\|^2,
\]
where
\[
c_n := \frac{1}{2\pi}\int^{2\pi}_0 e^{-in\theta} f(\theta) d\theta,\quad 
n \in \mathbb{Z}.
\]
Let $X$ be a Banach space. Let 
$T \in \mathcal{L}(X)$ and $S \in \mathcal{L}(X,Y)$.
Lemma~\ref{lem:T_power_rep} and Cauchy's theorem show that 
for all $k \in \mathbb{N}$,
$r > r(T)$, and $x \in X$, 
\[
\frac{1}{2\pi}\int^{2\pi}_0 e^{-in\theta} SR(re^{-i\theta},T)^kx d\theta
=
\begin{dcases}
	\binom{n-1}{k-1}
	\dfrac{ST^{n-k}x}{r^{n}}, & n \geq k, \\
	0, & n <k.
\end{dcases}
\]
Applying Parseval's equality to $f(\theta) \coloneqq SR(re^{-i\theta},T)^kx $, 
we obtain the following lemma,
which is also called Parseval's equality throughout this paper.
\begin{lemma}
	Let $X$ be a Banach space, and let $Y$ be a Hilbert space.
	Let $T \in \mathcal{L}(X)$ and $S \in \mathcal{L}(X,Y)$.
	For all $k \in \mathbb{N}$, $r > r(T)$, and $x \in X$,
	\[
	\frac{1}{2\pi}
	\int^{2\pi}_0 \|SR(re^{i\theta},T)^k x\|^2 d\theta =
	\sum_{n=0}^{\infty}
	\left\|
	\binom{n+k-1}{k-1}
	\frac{ST^nx}{r^{n+k}} \right\|^2.
	\]
\end{lemma}

\subsection{Special class of functions}
\subsubsection{Definition of $\alpha$-bounded regularly varying functions}
A measureable function $f\colon (0,\infty)
\to (0,\infty)$ is called {\em regularly varying of index $\alpha \in \mathbb{R}$} if
\[
\lim_{t \to \infty} \frac{f(\lambda t)}{f(t)} = \lambda^{\alpha}
\] 
for all $\lambda \geq 1$.
For simplicity, 
let $f\colon (0,\infty)
\to (0,\infty)$ be continuously differentiable.
If
\begin{equation}
	\label{eq:tf_der_f_limit}
	\lim_{t \to \infty} \frac{tf'(t)}{f(t)} = \alpha
\end{equation}
for some $\alpha \in \mathbb{R}$, then $f$ is regularly varying 
of index $\alpha$, and conversely
if $f$ is regularly varying of index $\alpha \in \mathbb{R}$
and $f'$ is non-decreasing or non-increasing, 
then \eqref{eq:tf_der_f_limit} holds; see, e.g., 
\cite[Proposition~0.7]{Resnick1987}.
Motivating this property of regularly varying functions, 
we consider the following variant of regularly varying functions.
\begin{definition}
	Let $\alpha \geq 0$.
	A function 
	$f \colon (0,\infty) \to (0,\infty)$ is called {\em $\alpha$-bounded
		regularly varying}
	if there exists a constant $t_0>0$ such that 
	$f$ is non-decreasing on $[t_0,\infty)$, 
	is absolutely continuous on each compact subinterval
	of $[t_0,\infty)$, and satisfies
	\[
	\frac{tf'(t)}{f(t)} \leq \alpha
	\] 
	for a.e.~$t \geq  t_0$.
\end{definition}

Functions of positive increase introduced in \cite{Rozendaal2019} are related to
$\alpha$-bounded regularly varying functions.
We say that a measurable function $f\colon (0,\infty) \to (0,\infty)$ 
has {\em positive increase} if there exist constants $\alpha >0$, $c\in (0,1]$,
and $t_0 >0$ such that 
\[
\frac{f(\lambda t)}{f(t)} \geq c \lambda^{\alpha}
\]
for all $\lambda \geq 1$ and $t \geq t_0$.
For simplicity, here we again assume that 
$f$ is continuously differentiable on $(0,\infty)$, and 
define 
\begin{equation}
	\label{eq:phi_def}
	\phi(t) \coloneqq \frac{tf'(t)}{f(t)}, \quad t >0.
\end{equation}
Then
\begin{equation}
	\label{eq:reg_rep}
	f(t) = f(1) \exp \left(
	\int^t_{1} \frac{\phi(s)}{s}ds
	\right)
\end{equation}
for all $t >0$.
Therefore, if
\[
\liminf_{t \to \infty } \frac{tf'(t)}{f(t)} =
\liminf_{t \to \infty } \phi(t) >0,
\]
then there exist $\alpha>0$
and $t_0 >0$ such that 
\[
\frac{f(\lambda t)}{f(t)} = \exp \left(
\int^{\lambda t}_{t} \frac{\phi(s)}{s}ds
\right) \geq \lambda^{\alpha}
\]
for all $\lambda \geq 1$ and $t \geq t_0$, and hence $f$
has positive increase.

There exist functions that are 
$\alpha$-bounded
regularly varying but are neither regularly varying nor have positive increase.

\begin{example}
	Let $f\colon (0,\infty) \to (0,\infty)$ be absolutely continuous
	on each compact subinterval of $(0,\infty)$, and 
	let $\phi$ be as in \eqref{eq:phi_def}. Suppose that $\phi$ can be written as 
	\[
	\phi(t) =
	\begin{cases}
		0, & 2 ^{n^2} \leq t <2^{(n+1)^2-1}, \\
		1, & 2^{(n+1)^2-1} \leq t <2^{(n+1)^2}
	\end{cases}
	\]
	for $n \in \mathbb{N}_0$.
	Then $f$ is $1$-bounded regularly varying.
	Let $\lambda \geq 1$ and take $n_0 \in \mathbb{N}$
	such that $2^{2n_0} \geq \lambda$.
	If $t = 2^{n^2}$ for some $n \geq n_0$, then
	$\lambda t \leq 2^{(n+1)^2-1}$ and hence
	\[
	\exp \left(
	\int^{\lambda t}_{t} \frac{\phi(s)}{s}ds
	\right) = 1.
	\]
	This and \eqref{eq:reg_rep} yield
	\begin{equation}
		\label{eq:ex_liminf}
		\liminf_{t\to \infty} \frac{f(\lambda t)}{f(t)} = 1.
	\end{equation}
	From \cite[Lemma~2.1]{Rozendaal2019}, we see that 
	$f$ does not have positive increase.
	On the other hand, if $t = 2^{(n+1)^2-1}$ for some $n \in \mathbb{N}_0$, then for all
	$\lambda \geq 2$,
	\[
	\exp \left(
	\int^{\lambda t}_{t} \frac{\phi(s)}{s}ds
	\right) \geq 2,
	\]
	and hence by \eqref{eq:reg_rep},
	\begin{equation}
		\label{eq:ex_limsup}
		\limsup_{t \to \infty} \frac{f(\lambda t)}{f(t)} \geq 2.
	\end{equation}
	From
	\eqref{eq:ex_liminf} and \eqref{eq:ex_limsup}, we see that
	the limit of $f(\lambda t)/f(t) $ as $t \to \infty$ does not exist for $\lambda \geq 2$. Hence, $f$ is not regularly varying.
\end{example}

\begin{remark}
	Let $a \geq 0$ and let $f\colon (a,\infty) \to (0,\infty)$
	be continuously differentiable such that $f'(t)>0$
	for all $t > a$.
	Then
	\[
	\limsup_{t \to \infty} \frac{tf'(t)}{f(t)} < \infty
	\quad \text{or equivalently}\quad 
	\liminf_{t \to \infty} \frac{f(t)}{tf'(t)} >0
	\]
	appears as a necessary condition for the function $f$
	to be of class $P_1$ in \cite[Lemma~2.2(iv)]{Eisner2021}. This condition was used there in the proof of a subsequential ergodic theorem.
\end{remark}

\subsubsection{Properties of $\alpha$-bounded regularly
	varying functions}
We begin with simple properties of $\alpha$-bounded regularly varying functions.
\begin{lemma}
	\label{lem:basic_property}
	Let 
	$f \colon (0,\infty) \to (0,\infty)$ be $\alpha$-bounded
	regularly varying for some $\alpha \geq 0$.
	Then the following statements hold:
	\begin{enumerate}[label=\upshape\alph*), leftmargin=*, widest=b]
		\item $f(t) = O(t^{\alpha})$ as $t \to \infty$.
		\item Let $\gamma >0$, and define $h(t) \coloneqq f(t)^\gamma$
		for $t >0$. Then $h$ is $(\alpha\gamma)$-bounded
		regularly varying.
	\end{enumerate}
\end{lemma}
\begin{proof}
	Since $f$ is $\alpha$-bounded
	regularly varying,
	there exists $t_0>0$ 
	such that 
	\[
	\frac{tf'(t)}{f(t)} \leq \alpha
	\] 
	for a.e.~$t \geq t_0$.
	
	a) Define $g(t) \coloneqq f(t)/t^\alpha$ for $t >0$. 
	Then 
	\[
	g'(t) =  \frac{tf'(t) - \alpha f(t)}{t^{\alpha+1}} \leq 0
	\]
	for a.e.~$t \geq t_0$.
	Hence, $g(t) \leq g(t_0)$ for all $t \geq t_0$.
	This implies that 
	$f(t) \leq f(t_0) (t/t_0)^{\alpha}$ for all $t \geq t_0$.
	
	b) Since
	\[
	\frac{th'(t)}{h(t)} = \gamma \frac{tf'(t)}{f(t)} \leq \alpha \gamma
	\]
	for a.e.~$t \geq t_0$, we see that $h$ is $(\alpha\gamma)$-bounded
	regularly varying.
\end{proof}
The following proposition will be used 
to characterize rates of decay for discrete operator semigroups.

\begin{proposition}
	\label{prop:cn_bound}
	Let $\alpha \geq 0$ and $\beta > \alpha - 1$.
	Let 
	$f \colon (0,\infty) \to (0,\infty)$ be  $\alpha$-bounded
	regularly varying.
	If 
	$c \colon \mathbb{N}_0 \to (0,\infty)$ satisfies
	\[
	c(n) = O \left(
	\frac{n^{\beta}}{f(n)}
	\right)
	\]
	as $n \to \infty$, then
	\begin{equation}
		\label{eq:cn_sum_bound}
		\sum_{n=0}^{\infty}
		\frac{c(n)}{r^n} = O 
		\left(
		\frac{1}{(r-1)^{\beta +1}f(1/(r-1))}
		\right)
	\end{equation}
	as $r \downarrow 1$.
\end{proposition}

The estimate \eqref{eq:cn_sum_bound}
in the case when $f(t) \equiv 1$ and $0 \leq \beta \leq 1$
can be obtained from \cite[Lemma~5.3]{Cohen2016}. 
To derive \eqref{eq:cn_sum_bound} in
the general case, we need 
the following integral estimate. This is an extension of
\cite[Lemma~4.2.a)]{Wakaiki2024JMAA}, which concerned
the case where $\beta=0$ and 
$f(t) = t^p (\log t)^q$
with $0\leq p < 1$ and $q \geq 0$.

\begin{lemma}
	\label{lem:int_bound}
	Let $\alpha \geq 0$ and $\beta > \alpha - 1$.
	If 
	$f \colon (0,\infty) \to (0,\infty)$ is  $\alpha$-bounded
	regularly varying,
	then
	there exist constants $M >0$ and $t_0 >0$ such that
	\[
	\int_{t_0}^{\infty} \frac{t^{\beta}e^{-s t}}{f(t)} dt \leq \frac{M}{s^{\beta+1} f(1/s)}
	\]
	for all $s \in(0, 1/t_0)$.
\end{lemma}
\begin{proof}
	Define $g (t) \coloneqq f(t)/t^\beta $ for $t >0$.
	By assumption,
	there exists $t_0 >0$ 
	such that  for a.e.~$t \geq  t_0$,
	\[
	\frac{tf'(t)}{f(t)} \leq \alpha.
	\]
	Hence,
	\begin{equation}
		\label{eq:g_pgc_delta}
		\frac{tg '(t)}{g (t)} = 
		\frac{tf'(t)}{f(t)} - \beta  \leq \alpha - \beta \eqqcolon \delta < 1
	\end{equation}
	for a.e.~$t \geq  t_0$.
	Let $s \in (0,1/t_0)$.
	Since $f$ is non-decreasing, it follows that 
	\begin{align}
		\int_{1/s}^{\infty} \frac{t^{\beta}e^{-s t}}{f(t)} dt 
		&\leq 
		\frac{1}{f(1/s)}
		\int_{1/s}^{\infty} t^{\beta}e^{-s t} dt \notag \\&=
		\frac{1}{s^{\beta+1}f(1/s)}
		\int_{1}^{\infty} t^{\beta}e^{-t} dt \notag \\&\leq 
		\frac{\Gamma(\beta+1)}{s^{\beta+1} f(1/s)}.
		\label{eq:tf_integral_right}
	\end{align}
	We also have
	\begin{equation}
		\label{eq:tf_integral_left1}
		\int^{1/s}_{t_0} \frac{t^{\beta}e^{-s t}}{f(t)} dt \leq 
		\int^{1/s}_{t_0} \frac{1}{g(t)} dt .
	\end{equation}
	Integration by parts yields
	\begin{align*}
		\int^{1/s}_{t_0} \frac{1}{g(t)} dt &=
		\int^{1/s}_{t_0} \frac{1}{t^{\delta}(g(t) /  t^{\delta})} dt \\
		&=
		\frac{1}{1-\delta}
		\left( 
		\frac{1/s}{g(1/s)} -
		\frac{t_0}{g(t_0)}  \right) + 	\frac{1}{1-\delta}
		\int^{1/s}_{t_0}
		\frac{t^{1-\delta}
			(t^{\delta} g'(t) - \delta t^{\delta-1} g(t))
		}{g(t)^2} dt\\
		&\leq 
		\frac{1}{(1-\delta)s g(1/s)} + \frac{1}{1-\delta}
		\int^{1/s}_{t_0}\left(
		\frac{
			tg'(t)
		}{g(t)^2} - 
		\frac{\delta}{g(t)} \right)dt.
	\end{align*}
	Since \eqref{eq:g_pgc_delta} gives
	\[
	\frac{
		tg'(t)
	}{g(t)^2} - 
	\frac{\delta}{g(t)} \leq \frac{\delta}{g(t)} - 	\frac{\delta}{g(t)} = 0
	\]
	for a.e.~$t \geq t_0$,
	it follows that 
	\begin{equation}
		\label{eq:tf_integral_left2}
		\int^{1/s}_{t_0} \frac{1}{g(t)} dt  \leq 	\frac{1}{(1-\delta) s g(1/s)}.
	\end{equation}
	Combining the estimates \eqref{eq:tf_integral_left1} and \eqref{eq:tf_integral_left2},
	we derive
	\begin{equation}
		\label{eq:tf_integral_left3}
		\int_{t_0}^{1/s} \frac{t^{\beta}e^{-s t}}{f(t)} dt  \leq 
		\frac{1}{(1-\delta)s^{\beta+1} f(1/s)}.
	\end{equation}
	By the estimate \eqref{eq:tf_integral_right} and \eqref{eq:tf_integral_left3},
	\[
	\int_{t_0}^{\infty} \frac{t^{\beta}e^{-s t}}{f(t)} dt \leq 
	\frac{M}{s^{\beta+1} f(1/s)},\quad
	\text{where~}M \coloneqq 
	\Gamma(\beta+1) + \frac{1}{1-\delta}.
	\] 
	Note that the constant $M$ does not depend on $s$.
\end{proof}

We are ready to prove Proposition~\ref{prop:cn_bound}.
\begin{proof}[Proof of Proposition~\ref{prop:cn_bound}.]
	By assumption, there exist $M_0>0$ and $n_0 \in \mathbb{N}$
	such that
	\begin{equation}
		\label{eq:cn_bound}
		c(n) \leq \frac{M_0n^{\beta}}{f(n)}
	\end{equation}
	for all $n \geq n_0$.  
	By Lemma~\ref{lem:int_bound},
	there exist $M_1 >0$ and $t_0 >0$ such that
	\begin{equation}
		\label{eq:t1_infty_int_estimate}
		\int_{t_0}^{\infty} \frac{t^{\beta}e^{-s t}}{f(t)} dt \leq \frac{M_1}{s^{\beta+1} f(1/s)} 
	\end{equation}
	for all $s \in (0,1/t_0)$.
	Let $n_1 \in \mathbb{N}$ satisfy $n_1 \geq \max \{ n_0,\,t_0 \}$.
	We have
	\begin{equation}
		\label{eq:cn_sum1}
		\sum_{n=0}^{n_1} 
		\frac{c(n)}{r^n} \leq 
		(n_1+1) \max_{0\leq n \leq n_1} c(n)
	\end{equation}
	for all $r > 1$.
	If $n \leq t \leq n+1$ for some $n \in \mathbb{N}$, then
	\[
	\frac{(n+1)^{\beta}}{r^{n+1}f(n+1) } \leq 
	\frac{(t+1)^{\beta}}{r^{t}f(t) }
	\]
	for all $r > 1$.
	From
	this and the estimate \eqref{eq:cn_bound}, we obtain
	\begin{equation}
		\label{eq:c_r_bound}
		\sum_{n=n_1+1}^{\infty} 
		\frac{c(n)}{r^n}
		\leq  M_0 
		\sum_{n=n_1}^{\infty} 
		\frac{(n+1)^{\beta}}{r^{n+1} f(n+1)} \leq 
		M_0
		\int_{n_1}^{\infty}\frac{(t+1)^{\beta}}{r^t f(t)} dt
	\end{equation}
	for all $r > 1$.
	Moreover, if $0 < \log r < 1/n_1 $, then the estimate~\eqref{eq:t1_infty_int_estimate} 
	gives
	\begin{align}
		\int_{n_1}^{\infty}\frac{(t+1)^{\beta}}{r^tf(t)} dt
		&\leq \left( \frac{n_1+1}{n_1} \right)^{\beta}
		\int_{n_1}^{\infty}\frac{t^{\beta}e^{-(\log r) t}}{f(t)} dt \notag \\
		&\leq M_1\left( \frac{n_1+1}{n_1} \right)^{\beta}
		\frac{1}
		{
			(\log r)^{\beta+1}
			f(1/\log r)
		}.
		\label{eq:int_t_rf}
	\end{align}
	Since 
	\[
	\frac{1}{2} \leq \frac{\log r}{r-1} \leq 1
	\]
	for all $r \in (1,3)$, we deduce from the estimates~\eqref{eq:c_r_bound} and \eqref{eq:int_t_rf} that
	\begin{equation}
		\label{eq:cn_sum2}
		\sum_{n=n_1+1}^{\infty} 
		\frac{c(n)}{r^n} \leq  M_0M_1\left( \frac{n_1+1}{n_1} \right)^{\beta}
		\frac{2^{\beta+1}}{(r - 1)^{\beta+1} f(1/(r- 1))}
	\end{equation}
	whenever $r \in (1,e^{1/n_1})$. 
	By Lemma~\ref{lem:basic_property}.a),
	$f(t) = O(t^{\alpha})$ as $t \to \infty$, and
	by assumption, $\beta+1 > \alpha$.
	Therefore,
	$
	s^{\beta+1}f(1/s)  \to 0
	$
	as $s \downarrow 0$. Thus,
	the estimate \eqref{eq:cn_sum_bound} follows from
	\eqref{eq:cn_sum1} and \eqref{eq:cn_sum2}.
\end{proof}

\section{Relation between 
	semigroup decay and resolvent growth}
\label{sec:resolvent_growth}

In this section, we study the relation between decay rates 
of discrete operator semigroups and growth rates of
resolvents.
First, we state and prove the main theorem of this section.
Next, we give some examples
showing the obtained estimates
cannot in general be improved.
Finally, we investigate the decay rate of $\|T^n(I-T)\|$ as a special case.
\subsection{Main result}
\label{sec:main_result}
The following theorem shows that 
rates of decay of discrete operator semigroups on Hilbert spaces
can be characterized in terms of rates of growth of 
resolvents in
a neighborhood of  $\mathbb{T}$.
\begin{theorem}
	\label{thm:resol_charactrization}
	Let $X$ be a Hilbert space. Let 
	$T \in \mathcal{L}(X)$ be 
	power-bounded, and
	suppose that 
	$S \in \mathcal{L}(X)$ commutes with $T$.
	Let $\alpha > 0$ and $k \in \mathbb{N}$ satisfy
	$k > \alpha$.
	If 
	$f \colon (0,\infty) \to (0,\infty)$ is  $\alpha$-bounded
	regularly varying,
	then the following statements are equivalent: \vspace{3pt}
	\begin{enumerate}[label=\upshape(\roman*), leftmargin=*, widest=ii]
		\item $\displaystyle
		\|T^nS\| = O \left(
		\frac{1}{
			f(n)}
		\right)$ as $n \to \infty$. \vspace{3pt}
		\item 
		There exists a constant $M>0$ such that 
		\begin{equation}
			\label{eq:resol_k_bound}
			\|R(\lambda,T)^{k} S\| \leq 
			\frac{M}{(|\lambda|-1)^{k}f(1/(|\lambda|-1) )}
		\end{equation}
		for all $\lambda \in \mathbb{E}$.
	\end{enumerate}
\end{theorem}

Recall that 
if $T$ is a power-bounded operator on a Banach space, then
$T$ is a strong Kreiss operator, which 
gives a better decay rate of $\|R(\lambda,T)^kS\|$ as $|\lambda|
\to \infty$ than the estimate \eqref{eq:resol_k_bound};
see \eqref{eq:Kreiss_cond_intro}.
Hence, Theorem~\ref{thm:resol_charactrization} shows that 
the growth rate of $\|R(\lambda,T)^{k} S\|$ in
a neighborhood of  $\mathbb{T}$
characterizes the decay rate of $\|T^nS\|$.

Before proceeding with the proof, we recall a series representation of the powers of resolvents.
Let $T$ be a bounded linear  operator on a Banach space.
If $r(T) \leq 1$, then 	
\begin{equation}
	\label{eq:resol_series}
	R(\lambda,T) = 
	\sum_{n=0}^{\infty}
	\frac{T^n}{\lambda^{n+1}}
\end{equation}
for all $\lambda \in \mathbb{E}$;
see, e.g., \cite[Proposition~9.28.c)]{Batkai2017}.
Hence,
\begin{align}
	\label{eq:Rk_rep}
	R(\lambda,T)^{k} = 
	\frac{R(\lambda,T)^{(k-1)}}{(-1)^{k-1}(k-1)!}
	= 
	\sum_{n=0}^{\infty}
	\binom{n+k-1}{k-1}
	\frac{T^n }{\lambda^{n+k}}
\end{align}
for all $\lambda \in \mathbb{E}$ and $k \in \mathbb{N}$.
Using this representation, we first show that the semigroup estimate
is transferred to the resolvent estimate.
\begin{lemma}
	\label{lem:Tn_to_resol}
	Let $X$, $Y$, and $Z$ be Banach spaces. Let 
	$T\in \mathcal{L}(X)$ satisfy $r(T) \leq 1$, and let 
	$S_1 \in \mathcal{L}(X,Z)$ and $S_2 \in \mathcal{L}(Y,X)$.
	Suppose that $\alpha \geq  0$ and $k \in \mathbb{N}$
	satisfy $k > \alpha$, and 
	let
	$f \colon (0,\infty) \to (0,\infty)$ be  $\alpha$-bounded 
	regularly varying. If
	\begin{equation}
		\label{eq:S1TS2_bound}
		\|S_1T^nS_2\| = O \left(
		\frac{1}{
			f(n)}
		\right)
	\end{equation}
	as $n \to \infty$,
	then
	there exists a constant $M>0$ such that 
	\begin{equation}
		\label{eq:resol_S_bound1}
		\|S_1R(\lambda,T)^{k}S_2\| \leq
		\frac{M}{(|\lambda|-1)^{k}f(1/(|\lambda|-1) )}
	\end{equation}
	for all  $\lambda \in \mathbb{C}$ with $1 < |\lambda| < 2$.
\end{lemma}
\begin{proof}
	It follows from \eqref{eq:Rk_rep} that
	\begin{align*}
		S_1R(\lambda,T)^{k}S_2 = 
		\sum_{n=0}^{\infty}
		\binom{n+k-1}{k-1}
		\frac{S_1T^n S_2}{\lambda^{n+k}}.
	\end{align*}
	From this, we obtain
	\begin{align*}
		\|S_1R(\lambda,T)^{k}S_2\| &\leq 
		\sum_{n=0}^{\infty} 
		\frac{c(n)}{|\lambda|^n}
	\end{align*}
	for all $\lambda \in \mathbb{E}$, where
	\[
	c(n) \coloneqq 
	\binom{n+k-1}{k-1} \,\| S_1T^nS_2 \|,\quad n \in \mathbb{N}_0.
	\]
	The assumption \eqref{eq:S1TS2_bound} yields
	\[
	c(n) = O \left(
	\frac{n^{k-1}}{f(n)}
	\right)
	\]
	as $n \to \infty$.
	By
	Proposition~\ref{prop:cn_bound} with $\beta = k - 1
	> \alpha - 1$,
	there exists $M>0$ such that 
	the desired estimate \eqref{eq:resol_S_bound1} holds
	whenever $1 < |\lambda| < 2$.
\end{proof}

Next, we establish the converse transference. To this end,
we work in the Hilbert space setting for Parseval's equality and use
the commutativity of $S$ and $T$ under the assumption that $X = Y$.
The technique employed here originates from 
\cite[Theorem 4.7]{Batty2016}, which concerned rates of decay for
$C_0$-semigroups.
\begin{lemma}
	\label{lem:resol_to_Tn}
	Let $X$ be a Hilbert space. Let $T \in \mathcal{L}(X)$ be power-bounded, and suppose that
	$S \in \mathcal{L}(X)$ commutes with $T$.
	For a fixed $k \in \mathbb{N}$, if
	$F\colon (0,\infty) \to (0,\infty) $ satisfies
	\begin{equation}
		\label{eq:resol_ineq}
		\| R(\lambda, T)^kS \| \leq F(|\lambda| - 1)
	\end{equation}
	for all $\lambda \in \mathbb{E}$,
	then there exists a constant $M>0$ such that
	\begin{equation}
		\label{eq:T_power_bound}
		\|T^n S\| \leq \frac{M F(1/n)}{n^k}
	\end{equation}
	for all $n \in \mathbb{N}$.
\end{lemma}
\begin{proof}
	Let $x \in X$, $k \in \mathbb{N}$, and 
	$K \coloneqq \sup_{n \in \mathbb{N}_0} \|T^n\|$. By the inequality \eqref{eq:resol_ineq} and the commutativity of $S$,
	\[
	\| R(\lambda, T)^{k+1}Sx \| \leq \| R(\lambda, T)^kS \|\, \|R(\lambda,T)x\| \leq F(|\lambda| - 1)\|R(\lambda,T)x\|
	\]
	for all $\lambda \in \mathbb{E}$.
	This and
	Parseval's equality give
	\begin{align*}
		\int_0^{2\pi} \|R(re^{i \theta}, T)^{k+1} Sx\|^2 d\theta &\leq 
		F(r - 1)^2 \int_0^{2\pi} \|R(re^{i \theta},T)x\|^2 d\theta \\
		&\leq 
		2\pi F(r - 1)^2 
		\sum_{n=0}^{\infty} \frac{ \|T^nx\|^2 }{r^{2(n+1)}} \\
		&=
		2\pi K^2  \frac{F(r- 1)^2}{r^2-1}\|x\|^2
	\end{align*}
	for all $r > 1$.
	Using Parseval's equality again, we obtain
	\[
	\sum_{m=0}^{\infty}
	\left\|
	\binom{m+k}{k}
	\frac{T^mSx}{r^{m+k+1}} \right\|^2
	\leq K^2  \frac{F(r- 1)^2}{r^2-1}\|x\|^2
	\]
	for all $r > 1$.
	Hence,
	\begin{align*}
		\sum_{m=0}^{n} (m+1)^{2k}\|T^mS x\|^2 &\leq 
		(k! r^{(n+k+1)})^2 
		\sum_{m = 0}^n	\left\|
		\binom{m+k}{k}
		\frac{T^mSx}{r^{m+k+1}} \right\|^2 \\
		&\leq 
		(k! K r^{(n+k+1)})^2  \frac{F(r- 1)^2}{r^2-1}\|x\|^2
	\end{align*}
	for all $r > 1$ and $n \in \mathbb{N}$.
	If we fix $n \in \mathbb{N}$ and set
	\[
	r = 1+\frac{1}{n},
	\]
	then
	\[
	\sum_{m=0}^{n} (m+1)^{2k}\|T^mS x\|^2 \leq (k! K)^2
	\left(
	1 + \frac{1}{n}
	\right)^{2(n+k+1)} \frac{n^2 F\left(
		1/n
		\right )^2}{2n+1} \|x\|^2.
	\]
	Since
	\[
	\left(
	1 + \frac{1}{n}
	\right)^{2(n+k+1)} \leq 2^{2(k+2)}
	\]
	for all $n \in \mathbb{N}$,
	we have
	\begin{equation}
		\label{kT_kth_power}
		\sum_{m=0}^{n} (m+1)^{2k}\|T^mS x\|^2 \leq M_1 n F\left(
		1/n
		\right )^2 \|x\|^2
	\end{equation}
	for all $n \in \mathbb{N}$, where
	$M_1 \coloneqq (k! K)^2 2^{2k+3}$.
	
	Define 
	\[
	p_k(n) \coloneqq \sum_{m=0}^n (m+1)^k,\quad n\in \mathbb{N}.
	\]
	Let $n \in \mathbb{N}$ and $y \in X$.
	Since
	\[
	T^n Sx = \frac{1}{p_k(n)} \sum_{m=0}^n (m+1)^k T^{n-m}T^m Sx,
	\]
	the Cauchy--Schwarz inequality yields
	\begin{align}
		\label{eq:TnSxy_bound}
		|\langle
		T^nSx, y
		\rangle | 
		&=
		\frac{1}{p_k(n)} 
		\left| 
		\sum_{m=0}^n (m+1)^k
		\langle 
		T^m Sx, (T^*)^{n-m}y
		\rangle
		\right| 
		\\
		&\leq 
		\frac{1}{p_k(n)} 
		\left( \sum_{m=0}^n (m+1)^{2k}
		\|
		T^m Sx
		\|^2\right)^{1/2}
		\left( 
		\sum_{m=0}^n 
		\|
		(T^*)^{n-m}y
		\|^2\right)^{1/2}. \notag 
	\end{align}
	From $\sup_{n \in \mathbb{N}_0} \|(T^*)^n\| = 
	\sup_{n \in \mathbb{N}_0} \|T^n\| = K$,
	we obtain
	\begin{equation}
		\label{eq:y_sum_bound}
		\sum_{m=0}^n 
		\|
		(T^*)^{n-m}y
		\|^2 \leq K^2(n+1) \|y\|^2.
	\end{equation}
	By \eqref{kT_kth_power}--\eqref{eq:y_sum_bound},
	\[
	|\langle
	T^nSx, y
	\rangle|
	\leq  K \sqrt{ M_1 }\, 
	\frac{\sqrt{n(n+1)}F\left(
		1/n
		\right )}{p_k(n)} 
	\|x\|\, \|y\|.
	\]
	Since $p_k$ is a polynomial 
	of degree $k+1$
	with positive leading coefficient,
	the desired inequality \eqref{eq:T_power_bound} holds
	for some constant $M>0$.
\end{proof}

\begin{proof}[Proof of Theorem~\ref{thm:resol_charactrization}.]
	The implication (i) $\Rightarrow $ (ii) follows from
	Lemma~\ref{lem:Tn_to_resol} and the resolvent estimate 
	\eqref{eq:Kreiss_cond_intro} for power-bounded operators.
	Suppose that 
	statement (ii)  hold.
	Define
	\[
	F(s) \coloneqq \frac{M}{s^{k} f(1/s)},\quad s>0.
	\]
	Then
	\[
	\frac{F(1/n)}{n^{k}} = 
	\frac{M}{f(n)}
	\]
	for all $n \in \mathbb{N}$.
	Hence,
	Lemma~\ref{lem:resol_to_Tn} shows that 
	statement (i) holds.
\end{proof}

Lemma~\ref{lem:Tn_to_resol} with $f(n) = n$ yields an estimate 
for $R(\lambda,T)^k$ for $k \geq 2$. 
The following proposition gives 
an estimate for $R(\lambda,T)$ in the case where $f(n) = n (\log n)^{\alpha}$ 
for some $\alpha \geq 0$,
although Lemma~\ref{lem:resol_to_Tn} does not provide
the converse statement.
\begin{proposition}
	\label{prop:nlogn_case}
	Let $X$, $Y$, and $Z$ be Banach spaces, and let $\alpha \geq 0$.
	If $T\in \mathcal{L}(X)$ with $r(T) \leq 1$, 
	$S_1 \in \mathcal{L}(X,Z)$, and $S_2 \in \mathcal{L}(Y,X)$ satisfy
	\[
	\|S_1T^nS_2\| = O \left(
	\frac{1}{n (\log n)^{\alpha}}
	\right)
	\]
	as $n \to \infty$,
	then
	there exists a constant $M>0$ such that 
	\begin{equation}
		\label{eq:nlogn_estimate}
		\|S_1R(\lambda,T)S_2\| \leq
		MH_{\alpha}(|\lambda|-1)
	\end{equation}
	for all  $\lambda \in \mathbb{C}$ with $1< |\lambda| < 4/3$, where
	\begin{equation}
		\label{eq:Ha_def}
		H_{\alpha}(s) \coloneqq
		\begin{cases}
			|\log s|^{1-\alpha}, & 0 \leq \alpha < 1, \\
			\log |\log s|, & \alpha = 1, \\
			1, & \alpha > 1
		\end{cases}
	\end{equation}
	for $0 < s < e^{-1}$. 
\end{proposition}
\begin{proof}
	By \cite[Lemma~4.2.b)]{Wakaiki2024JMAA}, there exist 
	$M>0$ and $t_0 > e$ such that 
	\[
	\int_{t_0}^{\infty} \frac{e^{-st}}{t(\log t)^{\alpha}} dt \leq  M H_{\alpha}(s)
	\]
	for all $s \in (0,1/t_0)$, where $H_{\alpha}$ is defined by 
	\eqref{eq:Ha_def}. The same argument as in the proof of Proposition~\ref{prop:cn_bound} shows that 
	if 	
	$c \colon \mathbb{N}_0 \to (0,\infty)$ satisfies
	\[
	c(n) = O \left(
	\frac{1}{n (\log n)^{\alpha}}
	\right)
	\]
	as $n \to \infty$, then
	\[
	\sum_{n=0}^{\infty}
	\frac{c(n)}{r^n} = O 
	(
	H_{\alpha}(r-1)
	)
	\]
	as $r \downarrow 1$.
	Except that this fact is used instead of Proposition~\ref{prop:cn_bound},
	the rest of the proof proceeds in the same way as the proof of Lemma~\ref{lem:Tn_to_resol},  so we omit the details.
\end{proof}

\subsection{Examples}
\label{sec:examples}
The first example shows that the decay rate of $\|T^nS\|$
and the growth rate of $\|R(\lambda,T)^kS\|$ in
Theorem~\ref{thm:resol_charactrization} cannot in general be 
replaced by better ones in the case of polynomial 
decay and growth.
\begin{example}
	Let $X \coloneqq \ell^2(\mathbb{N})$ and $\alpha >0$.
	Define $T,S_{\alpha} \in \mathcal{L}(X)$ by
	\begin{align*}
		Tx &\coloneqq
		\left(
		\left(
		1 - \frac{1}{j}
		\right) x_{j}
		\right)_{j \in \mathbb{N}}\quad \text{and} \quad 
		S_{\alpha}x \coloneqq
		\left(
		\frac{x_j}{j^{\alpha}}
		\right)_{j \in \mathbb{N}},
	\end{align*}
	where $x = (x_j)_{j \in \mathbb{N}} \in X$.
	Then
	\begin{equation}
		\label{eq:TnS_lower_bound}
		\limsup_{n \to \infty} n^{\alpha}\|T^nS_{\alpha}\| >0
	\end{equation}
	and
	\begin{equation}
		\label{eq:TnS_upper_bound}
		\|T^nS_{\alpha}\| = O \left( \frac{1}{n^{\alpha}}\right) 
	\end{equation}
	as $n \to \infty$. Indeed, since
	\[
	\|T^nS_{\alpha}\| = \sup_{j \in \mathbb{N}} 
	\left|
	\frac{1}{j^{\alpha}}
	\left(
	1 - \frac{1}{j}
	\right)^n
	\right| \geq \frac{1}{n^{\alpha}}
	\left(
	1 - \frac{1}{n}
	\right)^n
	\]
	for all $n \in \mathbb{N}$,
	the estimate \eqref{eq:TnS_lower_bound} holds.
	Define $f_{n,\alpha}(s) \coloneqq s^{\alpha}(1-s)^n$ for $0 < s \leq 1$.
	Since
	\[
	f_{n,\alpha}'(s) = s^{\alpha - 1}(1-s)^{n-1}(\alpha-(n+\alpha)s),
	\]
	we obtain
	\[
	\|T^nS_{\alpha}\| \leq \sup_{0<s\leq 1} f_{n,\alpha}(s) =
	f_{n,\alpha}
	\left(
	\frac{\alpha}{n+\alpha}
	\right)
	=  \left(\frac{\alpha}{n+\alpha} \right)^{\alpha} \left(
	1 - \frac{\alpha}{n+\alpha}
	\right)^n
	\]
	for all $n \in \mathbb{N}$.
	This implies that the estimate \eqref{eq:TnS_upper_bound} holds.

	First, we consider the case $\alpha = 1/2$, and will show that 
	\begin{align}
		\label{eq:a_1_2_lower_bound}
		\limsup_{r \downarrow 1} \sqrt{r-1} \, \|R(r,T)S_{1/2}\| > 0
	\end{align}
	and that there exists a constant $M_1>0$ such that  
	\begin{align}
		\label{eq:a_1_2_upper_bound}
		\|R(\lambda,T)S_{1/2}\| \leq \frac{M_1}{\sqrt{|\lambda| - 1} }
	\end{align}
	for all $\lambda \in \mathbb{E}$.
	From \eqref{eq:TnS_lower_bound}--\eqref{eq:a_1_2_upper_bound},
	we see that 
	the estimates in statements (i) and (ii) of
	Theorem~\ref{thm:resol_charactrization} with $k = 1$
	cannot in general be improved in the polynomial case.
	For all $\lambda \in \mathbb{E}$, we have
	\begin{equation}
		\label{eq:bound_by_g_lam}
		\|R(\lambda,T)S_{1/2}\| = \sup_{j \in \mathbb{N}}
		\left|
		\frac{1}{\sqrt{j} (\lambda - (1-1/j)) }
		\right| =
		\sup_{j \in \mathbb{N}}
		\frac{1 }{|g_{\lambda}(j)|},
	\end{equation}
	where
	\[
	g_\lambda(s) \coloneqq \frac{\lambda s - s+1}{\sqrt{s}},\quad s \geq 1.
	\]
	If $\lambda = r\in (1,2)$, then
	\[
	\inf_{s \geq 1}|g_r(s)|  =
	g_r\left(
	\frac{1}{r-1}
	\right) = 2 \sqrt{r-1}.
	\]
	For $j \in \mathbb{N}$,
	define $r_j > 1$ by
	\[
	r_j \coloneqq \frac{j+1}{j}.
	\]
	Then 
	\[
	\frac{1}{r_j-1} = j\quad 
	\text{and} \quad 
	\lim_{j \to \infty} r_j = 1.
	\]
	Hence,
	the estimate \eqref{eq:a_1_2_lower_bound} holds.

	To prove the estimate \eqref{eq:a_1_2_upper_bound},
	let $\lambda = re^{i\theta}$ for some 
	$r \in (1,2)$ and $\theta \in [-\pi,\pi)$.
	Then
	\begin{equation}
		\label{eq:def_h_rt}
		|g_\lambda(s)|^2 = 
		\frac{(r^2-2 r\cos \theta + 1)s^2 + 2(r \cos \theta - 1)s + 1}{s} \eqqcolon
		h_{r,\theta}(s)
	\end{equation}
	for all $s \geq 1$.
	The derivative $h_{r,\theta}'$
	is given by
	\[
	h_{r,\theta}'(s) = \frac{(r^2 - 2r \cos \theta + 1)s^2 - 1}{s^2},
	\]
	and therefore we define
	\[
	s_{r,\theta} \coloneqq \frac{1}{\sqrt{r^2-2r\cos\theta + 1} }.
	\]
	Then, for all $s \geq 1$,
	\begin{equation}
		\label{eq:h_r_theta_lower_bound}
		h_{r,\theta}(s) \geq 
		\begin{cases}
			h_{r,\theta}(s_{r,\theta}) = \phi_r(\cos \theta),&
			r \leq 2 \cos \theta,
			\\
			h_{r,\theta}(1) = r^2, & r > 2 \cos \theta,
		\end{cases}
	\end{equation}
	where
	\[
	\phi_r(w) \coloneqq 
	2\sqrt{r^2-2rw + 1} + 2(r w - 1)
	\]
	for $w \in [r/2, 1]$. Since
	$\phi_r$ monotonically decreases on $[r/2,1]$, we obtain
	\begin{equation}
		\label{eq:phi_r_lower_bound}
		\phi_r(w) \geq \phi_r(1) = 4(r-1)
	\end{equation}
	for all $w \in [r/2 , 1]$.
	The estimates \eqref{eq:h_r_theta_lower_bound}
	and \eqref{eq:phi_r_lower_bound} yield
	\[
	h_{r, \theta}(s) \geq 
	\min\{
	4(r-1),\, r^2
	\}
	\]
	for all $s \geq 1$.
	Combining this with \eqref{eq:bound_by_g_lam} and \eqref{eq:def_h_rt},
	we see that 
	the estimate \eqref{eq:a_1_2_upper_bound} holds 
	whenever $1 < |\lambda| < 2$.
	The case $|\lambda| \geq 2$
	follows  from the power-boundedness of $T$; see 
	\eqref{eq:Kreiss_cond_intro}.

	Let $k \in \mathbb{N}$.
	In the case $\alpha = k/2$, we have that
	\[
	\|R(\lambda,T)^kS_{k/2}\| = \sup_{j \in \mathbb{N}}
	\left|
	\frac{1}{j^{k/2} (\lambda - (1-1/j))^k }
	\right| = \|R(\lambda,T)S_{1/2}\|^k.
	\]
	Therefore, \eqref{eq:a_1_2_lower_bound}
	and \eqref{eq:a_1_2_upper_bound} yield, respectively,
	\[
	\limsup_{r \downarrow 1}\, (r- 1)^{k/2} \|R(r,T)^kS_{k/2}\| >0
	\]
	and 
	\[
	\|R(\lambda,T)^kS_{k/2}\| \leq \frac{M_k}{(|\lambda| - 1)^{k/2}}
	\]
	for all $\lambda \in \mathbb{E}$ and some constant $M_k >0$.
	Hence, for each $k \in \mathbb{N}$, 
	the estimates in statements (i) and (ii) of
	Theorem~\ref{thm:resol_charactrization}
	cannot in general be improved in the polynomial case.
\end{example}

Let $X$ and $Y$ be Banach spaces.
Let $T\in \mathcal{L}(X)$ be power-bounded, 
and let $S \in \mathcal{L}(X,Y)$. Assume that 
$\|ST^n\| = O(n^{-1})$ as $n \to \infty$.
By Lemma~\ref{lem:Tn_to_resol},
there exists a constant $M_1>0$ such that 
$\|SR(\lambda,T)^2\|\leq M_1/ (|\lambda|-1)$ for all $\lambda \in
\mathbb{E}$.
Moreover, we see from Proposition~\ref{prop:nlogn_case} that
$\|SR(\lambda,T)\|\leq M_2 \log(|\lambda|-1)$ for all $\lambda \in
\mathbb{E}$ and some constant $M_2>0$.
One may ask whether $\sup_{\lambda \in \mathbb{E}}\|SR(\lambda,T)\| < \infty$
holds in the Hilbert space setting.
The second example shows that the answer of this question
is negative.

\begin{example}
	Let $X \coloneqq \ell^2(\mathbb{N})$. Define $T\in \mathcal{L}(X)$ 
	and $S \in \mathcal{L}(X,\mathbb{C})$ by
	\[
	Tx \coloneqq
	\left(
	\left(
	1 - \frac{1}{\sqrt{j}}
	\right) x_j
	\right)_{j \in \mathbb{N}}\quad \text{and} \quad 
	Sx \coloneqq
	\sum_{j=1}^{\infty}
	\frac{x_j}{j},
	\]
	where $x = (x_j)_{j \in \mathbb{N}} \in X$.
	We will prove below that 
	there exist constants $K_1,K_2,M_1,M_2 >0$  such that 
	\begin{equation}
		\label{eq:STn_bounds}
		\frac{K_1}{n} \leq \|ST^n\| \leq \frac{K_2}{n}.
	\end{equation}
	for all $n \in \mathbb{N}$ and
	\begin{equation}
		\label{eq:SResol_bounds}
		M_1 \sqrt{|\log (r-1)|} \leq 
		\|SR(r,T)\| \leq M_2 \sqrt{|\log (r-1)|}
	\end{equation}
	for all $r \in (1,2)$.
	These estimates show that 
	$\|ST^n\| = O(n^{-1})$ as $n \to \infty$
	does not in general imply $\sup_{\lambda \in \mathbb{E}}\|SR(\lambda,T)\|< \infty$.
	
	Let $n \in \mathbb{N}$.
	Then
	\[
	\|ST^n\|^2 = 
	\sum_{j=2}^{\infty}
	\left|\frac{1}{j}
	\left(
	1 - \frac{1}{\sqrt{j}}
	\right)^n
	\right|^2.
	\]
	If $j \leq s < j+1$ for some $j  \in \mathbb{N}$
	with $j \geq 2$, 
	then
	\[
	\frac{1}{2(s-1)}
	\left(
	1 - \frac{1}{\sqrt{s-1}}
	\right)^n  \leq 
	\frac{1}{j} \left(
	1 - \frac{1}{\sqrt{j}}
	\right)^n 
	\leq 
	\frac{2}{s}\left(
	1 - \frac{1}{\sqrt{s}}
	\right)^n .
	\]
	Therefore,
	\begin{align}
		\label{eq:STn2_bounds}
		\frac{1}{4}
		\int_1^{\infty}  
		\frac{1}{s^2} 
		\left(
		1 - \frac{1}{\sqrt{s}}
		\right)^{2n}ds \leq 
		\|ST^n\|^2  \leq 
		4\int_2^{\infty}
		\frac{1}{s^2}  
		\left(
		1 - \frac{1}{\sqrt{s}}
		\right)^{2n}ds.
	\end{align}
	In addition, we have
	\begin{align}
		\label{eq:bounds_comp}
		\int_1^{\infty}
		\frac{1}{s^2}  
		\left(
		1 - \frac{1}{\sqrt{s}}
		\right)^{2n}ds 
		= 2
		\int_0^{1}  
		\tau
		\left(
		1 - \tau
		\right)^{2n}d\tau =
		\frac{1}{(n+1)(2n+1)}.
	\end{align}
	By
	\eqref{eq:STn2_bounds} and \eqref{eq:bounds_comp},
	there exist $K_1,K_2>0$ such that 
	the estimate \eqref{eq:STn_bounds} holds for all $n \in \mathbb{N}$.

	Let $r \in (1,2)$. We have
	\[
	\|SR(r,T)\|^2 
	= 
	\sum_{j=1}^{\infty}
	\frac{1}{((r-1)j + \sqrt{j})^2}.
	\]
	If $j \leq s < j+1$ for some $j \in \mathbb{N}$, then
	\[
	\frac{1}{(r-1)s + \sqrt{s} } 
	\leq 
	\frac{1}{(r-1)j + \sqrt{j}} \leq 
	\frac{2}{(r-1)s + \sqrt{s}}.
	\]
	Hence,
	\[
	\int_1^{\infty}
	\frac{1}{((r-1)s+\sqrt{s})^2} ds \leq 
	\|SR(r,T)\|^2  \leq 
	\int_1^{\infty}
	\frac{4}{((r-1)s+\sqrt{s})^2} ds .
	\]
	Since
	\begin{align*}
		\int_1^{\infty}
		\frac{1}{((r-1)s+\sqrt{s})^2} ds &=
		\int_1^{\infty}
		\frac{2}{((r-1)\tau+1)^2\tau} d\tau \\&=
		2 \left(\log r - \log (r-1) - \frac{1}{r} \right),
	\end{align*}
	there exist $M_1,M_2 >0$
	such that the estimate~\eqref{eq:SResol_bounds} holds
	for all $r \in (1,2)$.
\end{example}

Finally, we give an example showing that
the estimate in Proposition~\ref{prop:nlogn_case} cannot in general
improved.
\begin{example}
	Let $X \coloneqq c_0(\mathbb{N})$ and 
	let $T$ be a left shift operator on $X$, i.e., 
	$Tx = (x_{j+1})_{j \in \mathbb{N}}$ for $x = (x_{j})_{j \in \mathbb{N}} \in X$. Then
	$\|T\|\leq 1$ and $\lim_{n \to \infty}
	\|T^n x\| = 0$ for all $x \in X$.
	Let $b = (b_j)_{j \in \mathbb{N}}$ be a non-increasing
	sequence of positive real numbers.
	Define
	$S \in \mathcal{L}(X)$ by
	\[
	Sx \coloneqq (
	b_jx_j
	)_{j \in \mathbb{N}},
	\quad x = (x_j)_{j \in \mathbb{N}} \in X.
	\]
	Since
	\[
	T^nSx = 
	(
	b_{j+n} x_{j+n}
	)_{j \in \mathbb{N}},
	\quad x = (x_j)_{j \in \mathbb{N}} \in X
	\]
	for all $n \in \mathbb{N}$,
	it follows immediately that
	\[
	\|T^nS\| = O \left( b_n \right)
	\]
	as $n \to \infty$. 
	Let $\alpha \geq 0$. When 
	\begin{equation}
		\label{eqLbj_def}
		b_j = \frac{1}{j ( \log (j+1) )^{\alpha}},\quad j \in \mathbb{N},
	\end{equation} 
	there exists $c>0$ such that 
	\begin{equation}
		\label{eq:resol_S_lower_bound_ex}
		\|R(r,T)S \| \geq  cH_{\alpha}(r-1)
	\end{equation}
	for all $r \in (1,4/3)$, where  $H_{\alpha}$ is defined by 
	\eqref{eq:Ha_def}.
	
	To prove \eqref{eq:resol_S_lower_bound_ex}, 
	let $r \in (1,4/3)$ and
	define
	\[
	x^m = (x^m_j)_{j \in \mathbb{N}}
	\coloneqq
	\left( 
	\frac{m}{j+m}
	\right)_{j \in \mathbb{N}}
	\]
	for $m \in \mathbb{N}$.
	Then $x^m \in c_0(\mathbb{N})$ and
	$\|x^m\| \leq 1$ for all $m \in \mathbb{N}$.
	Moreover,
	\begin{equation}
		\label{eq:xjm_lim}
		\lim_{m \to \infty} x_j^m = 1
	\end{equation}
	for all $j \in \mathbb{N}$.
	Since \eqref{eq:resol_series} yields
	\[
	R(r,T)S = \sum_{n=0}^{\infty} \frac{T^n S}{r^{n+1}},
	\]
	the first entry of $R(r,T)S x^m$
	is given by
	\[
	\sum_{n=1}^{\infty} \frac{b_n x^m_n }{r^n
	}.
	\]
	Using \eqref{eq:xjm_lim} and 
	the dominated convergence theorem, we obtain
	\[
	\lim_{m \to \infty}\sum_{n=1}^{\infty} \frac{b_n x^m_n }{r^n } = 
	\sum_{n=1}^{\infty} \frac{b_n}{r^n}.
	\]
	Let $b =(b_j)_{j \in \mathbb{N}}$ be given by \eqref{eqLbj_def}
	for some $\alpha \geq 0$.
	Then
	\[
	\sum_{n=1}^{\infty} \frac{b_n}{r^n } 
	\geq 
	\int_1^{\infty} \frac{e^{-(\log r)t }}{t (\log (t+1))^{\alpha}} dt
	\geq 
	\int_2^{\infty} \frac{e^{-(\log r)t }}{t (\log t)^{\alpha}} dt.
	\]
	Moreover,
	\begin{align*}
		\int_2^{\infty} \frac{e^{-(\log r)t }}{t (\log t)^{\alpha}} dt
		\geq 
		\int_2^{1/\log r} \frac{e^{-(\log r)t }}{t (\log t)^{\alpha}} dt 
		\geq 
		e^{-1}
		\int_2^{1/\log r} \frac{1}{t (\log t)^{\alpha}} dt.
	\end{align*}
	We have
	\[
	e^{-1}\int_2^{1/\log r} \frac{1}{t (\log t)^{\alpha}} dt \geq c H_{\alpha}(\log r)
	\]
	for some $c >0$ independent of $r$.
	Since $\log r \leq r-1$, we obtain
	$
	H_{\alpha}(\log r) \geq 
	H_{\alpha}(r-1).
	$
	Therefore,
	\[
	\sup_{m \in \mathbb{N}} \|R(r,T)Sx^m\| \geq 
	c H_{\alpha}(r-1).
	\]
	Recalling that $\|x^m\| \leq 1$ for all $m \in \mathbb{N}$, we conclude that  \eqref{eq:resol_S_lower_bound_ex} holds for all 
	$r \in (1,4/3)$.
\end{example}

\subsection{The special case when \texorpdfstring{\(\bm{S=I-T}\)}{S=I-T}}
\label{sec:special_case}
Now we focus on the case when
$S = I-T$. 
In this subsection, we often use the following simple relation:
Let $T$ be a bounded linear operator on a Banach space. Then
\begin{equation}
	\label{eq:resolvent_change}
	R(\lambda,T)(I-T) = 1 - (\lambda-1) R(\lambda,T)
\end{equation}
for all $\lambda \in \varrho(T)$.

We recall the following result on Ritt  operators; see 
\cite{Nagy1999, Lyubich1999} for the proof.
\begin{theorem}
	\label{thm:Ritt_characterization}
	Let $X$ be a Banach space, and let $T \in \mathcal{L}(X)$.
	Then the following statements are equivalent:
	\begin{enumerate}[label=\upshape(\roman*), leftmargin=*, widest=ii]
		\item $T$ is a Ritt operator.
		\item $T$ is power-bounded, and $\|T^n(I-T)\| = O(n^{-1})$
		as $n \to \infty$.
	\end{enumerate}
\end{theorem}

We characterize Ritt operators on Hilbert spaces in terms of the powers of resolvents
as a corollary of Theorem~\ref{thm:resol_charactrization}.
\begin{corollary}
	\label{coro:Ritt1}
	Let $X$ be a Hilbert space. Let $T \in \mathcal{L}(X)$
	and $k \in \mathbb{N}$.
	Then the following statements are equivalent:
	\begin{enumerate}[label=\upshape(\roman*), leftmargin=*, widest=ii]
		\item $T$ is a Ritt operator.
		\item $T$ is power-bounded, and there exists 
		a constant $M>0$ such that 
		\begin{equation}
			\label{eq:resolvent_ritt_power}
			\|R(\lambda,T)^{k+1}\| \leq \frac{M}{ |\lambda - 1|(|\lambda|-1)^k}
		\end{equation}
		for all $\lambda \in \mathbb{E}$.
	\end{enumerate}
\end{corollary}

\begin{proof}
	Suppose that statement~(i) holds. By Theorem~\ref{thm:Ritt_characterization},
	$T$ is power-bounded. The estimate~\eqref{eq:resolvent_ritt_power}
	follows immediately from the definition of Ritt operators and 
	\[
	\frac{1}{|\lambda - 1|^{k+1}} \leq 
	\frac{1}{ |\lambda - 1|(|\lambda|-1)^{k}} 
	\]
	for all $\lambda \in \mathbb{E}$.
	
	Suppose that statement~(ii) holds.
	Since $T$ is a Kreiss operator,
	there exists $M_1 >0$ such that 
	\begin{equation}
		\label{eq:R_k_bound}
		\|
		R(\lambda,T)^k 
		\|\leq 
		\frac{M_1}{(|\lambda| - 1)^k}
	\end{equation}
	for all 
	$\lambda \in \mathbb{E}$.
	From \eqref{eq:resolvent_change}--\eqref{eq:R_k_bound},
	we obtain
	\begin{align*}
		\|
		R(\lambda,T)^{k+1} (I-T)
		\| 
		& \leq 
		\|
		R(\lambda,T)^k
		\| +
		\|
		(\lambda-1) R(\lambda,T)^{k+1}
		\| 
		\\
		&\leq 
		\frac{M_1+M}{(|\lambda| - 1)^k}
	\end{align*}
	for all $\lambda \in \mathbb{E}$.
	Theorem~\ref{thm:resol_charactrization}
	shows that  
	$\|T^n (I-T)\| = O (n^{-1})$ as $n \to \infty$.
	Thus, statement~(i) holds
	by Theorem~\ref{thm:Ritt_characterization}.
\end{proof}

For $c,\delta \geq 1$, define
\begin{equation}
	\label{eq:Stolz_domain}
	S^{\delta}_{c} \coloneqq \left\{
	\lambda \in \mathbb{D} :
	\frac{|1-\lambda|^{\delta}}{1 - |\lambda|} < c
	\right\} \cup \{ 1\},
\end{equation}
Following \cite{Mahillo2024}, we 
call $S_c^{\delta}$ the {\em $\delta$-Stolz domain}.
It is worth noting that 
a closely related concept was previously introduced
in \cite{Badea2017}, where
it was referred to as a {\em generalised Stolz domain with parameter $\delta$}.
The $(3/2)$-Stolz domain is illustrated in \cite[Fig.~3]{Mahillo2024}.
The $1$-Stolz domain coincides with
the Stolz domain studied in \cite{Gomilko2018}.

By \cite[Proposition~2.6]{Cohen2016}, 
if a power-bounded operator $T$ on a Banach space 
satisfies $\|T^n(I-T)\| = O(n^{-\alpha})$ as $n \to \infty$ for some $\alpha \in (1/2,1)$, then
$\sigma(T)$ is contained in a certain quasi-Stolz set.
Here we show that 
a similar implication is true for $\alpha \in (0,1)$,
by using $(1/\alpha)$-Stolz domains instead of quasi-Stolz sets.
A short 
discussion of a comparison between quasi-Stolz sets 
and $\delta$-Stolz domains can be found in
\cite[Remark~2.7]{Mahillo2024}.
\begin{lemma}
	\label{lem:Banach_RK_cond}
	Let $X$ be a Banach space, and let $T \in \mathcal{L}(X)$
	satisfy $r(T) \leq 1$.
	Then the following statements hold:
	\begin{enumerate}[label=\upshape\alph*), leftmargin=*, widest=b]
		\item  If 
		$\|T^n(I-T)\| = O(n^{-\alpha})$
		as $n \to \infty$ for some $\alpha \in [0,1)$,
		there exists a constant $M>0$ such that 
		\begin{equation}
			\label{eq:I-T_resol1}
			\|R(\lambda,T)\| \leq 
			\frac{M}{|\lambda-1|(|\lambda|-1)^{1-\alpha}}
		\end{equation}
		for all $\lambda \in \mathbb{C}$ with $1 < |\lambda| < 2$.
		\item Let $\alpha \in (0,1)$.
		If there exists a constant $M>0$ such that 
		$T$ satisfies
		\eqref{eq:I-T_resol1}  for all $\lambda \in \mathbb{C}$ with $1 < |\lambda| < 2$ and 
		\begin{equation}
			\label{eq:I-T_resol2}
			\|R(\lambda,T)\| \leq \frac{M}{|\lambda| - 1}
		\end{equation}
		for all $\lambda \in \mathbb{E}$, then
		\begin{equation}
			\label{eq:spectrum_inclusion}
			\sigma(T) \subset \overline{S^{1/\alpha}_{c}}
		\end{equation}
		for some constant $c \geq 2$, where 
		$S^{1/\alpha}_{c}$ is the $(1/\alpha)$-Stolz domain defined by \eqref{eq:Stolz_domain}.
	\end{enumerate}
\end{lemma}
\begin{proof}
	a) 
	By Lemma~\ref{lem:Tn_to_resol}, 
	there exists $M_1>0$ such that 
	\[
	\|R(\lambda,T)(I-T)\| \leq 
	\frac{M_1}{(|\lambda|-1)^{1-\alpha}}
	\]
	for all $\lambda \in \mathbb{C}$ with $1 < |\lambda| < 2$. Combining this with
	\eqref{eq:resolvent_change}, we obtain
	\[
	\|R(\lambda,T)\| \leq \frac{1}{|\lambda- 1|} + \frac{M_1}{|\lambda- 1|(|\lambda|-1)^{1-\alpha}}
	\leq \frac{1+M_1}{|\lambda- 1|(|\lambda|-1)^{1-\alpha}}
	\]
	for all $\lambda \in \mathbb{C}$ with $1 < |\lambda| < 2$.
	
	b)
	By the estimates~\eqref{eq:I-T_resol1} and 
	\eqref{eq:I-T_resol2},
	there exists $M_1 >0$ such that 
	\begin{equation}
		\label{eq:RK_cond}
		\|R(\lambda,T)\| \leq \frac{M_1|\lambda|^{1-\alpha}}{ |\lambda- 1|(|\lambda| - 1)^{1-\alpha}}
	\end{equation}
	for all $\lambda \in \mathbb{E}$.
	Hence, \cite[Theorem~2.6]{Mahillo2024} shows that the desired inclusion
	\eqref{eq:spectrum_inclusion}
	holds.
\end{proof}

If $T$ is a power-bounded operator on a Hilbert space, then
the converse of Lemma~\ref{lem:Banach_RK_cond}.a)
also holds.
\begin{proposition}
	\label{prop:equiv_decay_RK}
	Let $X$ be a Hilbert space, and let $T \in \mathcal{L}(X)$
	be power-bounded.
	Then the following statements are equivalent for a 
	fixed $\alpha \in (0,1)$:
	\begin{enumerate}[label=\upshape(\roman*), leftmargin=*, widest=ii]
		\item $\|T^n(I-T)\| = O(n^{-\alpha})$ as $n \to \infty$.
		\item 		There exists a constant $M>0$ such that 
		\[
		\|R(\lambda,T)\| \leq 
		\frac{M}{|\lambda-1|(|\lambda|-1)^{1-\alpha}}
		\]
		for all $\lambda \in \mathbb{C}$ with $1 < |\lambda| < 2$.
	\end{enumerate}
\end{proposition}
\begin{proof}
	The implication (i) $\Rightarrow$ (ii) follows from Lemma~\ref{lem:Banach_RK_cond}.a).
	Suppose that statement (ii) holds. Then
	\eqref{eq:resolvent_change} yields
	\[
	\|R(\lambda,T)(I-T)\| \leq 
	\frac{1+M}{(|\lambda|-1)^{1-\alpha}} 
	\]
	for all $\lambda \in \mathbb{C}$ with $1 < |\lambda| < 2$.
	Combining this with the fact that  $T$ is a Kreiss operator, 
	we see that there exists $M_1 >0$ such that 
	\[
	\|R(\lambda,T)(I-T)\| \leq 
	\frac{M_1}{(|\lambda|-1)^{1-\alpha}} 
	\]
	for all $\lambda \in \mathbb{E}$. Hence,
	Theorem~\ref{thm:resol_charactrization} shows that 
	$\|T^n(I-T)\| = O(n^{-\alpha})$ as $n \to \infty$.
\end{proof}

\begin{remark}
	The implication (ii) $\Rightarrow$ (i) of Proposition~\ref{prop:equiv_decay_RK}
	can be proved  alternatively
	by using the results in \cite{Seifert2016,Mahillo2024}.
	Let $\alpha ,\beta \geq 0$.
	Following \cite[Definition~1.1]{Mahillo2024},
	we say that 
	$T \in \mathcal{L}(X)$ is an
	{\em $(\alpha,\beta)$-RK operator}
	if $T$ satisfies
	$r(T) \leq 1$
	and
	\[
	\|R(\lambda, T)\| \leq \frac{M |\lambda|^{\alpha+\beta-1}}{|\lambda-1|^{\alpha} (|\lambda| - 1)^{\beta}}
	\]
	for all $\lambda \in \mathbb{E}$ and some 
	constant $M>0$. 
	As seen in the proof of Lemma~\ref{lem:Banach_RK_cond}.b),
	if
	a power-bounded operator $T$ on a Banach space $X$ 
	satisfies the estimate given in 
	statement (ii) of Proposition~\ref{prop:equiv_decay_RK},
	then $T$ is a
	$(1,1-\alpha)$-RK operator. It follows from \cite[Lemma~2.1]{Mahillo2024}
	that
	$\sigma(T) \cap \mathbb{T}
	\subset \{1 \}$. Moreover, 
	\begin{equation}
		\label{eq:resolvent_bound_1_a}
		\|R(e^{i \theta},T)\| = O \left(
		\frac{1}{|\theta|^{1/\alpha}}
		\right)
	\end{equation}
	as $\theta \to 0$
	by  \cite[Proposition~2.2]{Mahillo2024}.
	Therefore, \cite[Corollary~3.1]{Seifert2016} shows that
	\[
	\|T^n(I-T)\| = O \left( \frac{(\log n)^{\alpha} }{n^{\alpha}}\right)
	\]
	as $n \to \infty$; see also \cite[Theorem~3.5]{Mahillo2024}.
	When $X$ is a Hilbert space, 
	$\|T^n(I-T)\| = O(n^{-\alpha})$ as $n \to \infty$ by 
	\cite[Theorem~3.10]{Seifert2016}.
\end{remark}

Let $T$ be
a normal contraction on a Hilbert space, and let
$\alpha \in (1/2,1)$.
In \cite[Corollary~2.7]{Cohen2016},
the decay rate $\|T^n(I-T)\| = O(n^{-\alpha})$ as $n \to \infty$
was characterized in terms of the property that $\sigma(T)$
is contained in some quasi-Stolz set.
Using the $\delta$-Stolz domain and Lemma~\ref{lem:Banach_RK_cond}, we obtain
an extended characterization in the Banach space case with $\alpha \in (0,1)$.
\begin{proposition}
	\label{prop:spectrum_characterization}
	Let $X$ be a Banach space, and let
	$T \in \mathcal{L}(X)$ be a Kreiss  operator. Assume that 
	there exist constants $C \geq 1$ and $n_0 \in \mathbb{N}$
	such that 
	\begin{equation}
		\label{eq:spectrum_bound}
		\|T^n (I-T)\| \leq C \max_{\lambda \in \sigma(T)} |\lambda|^n |1-\lambda|
	\end{equation}
	for all $n \geq n_0$.
	Then the following statements are equivalent for a fixed 
	$\alpha \in (0,1)$:
	\begin{enumerate}[label=\upshape(\roman*), leftmargin=*, widest=iii]
		\item $\|T^n(I-T)\| = O(n^{-\alpha} )$ as $n \to \infty$.
		\item  There exists a constant $M>0$ such that 
		\[
		\|R(\lambda,T)\| \leq 
		\frac{M}{|\lambda-1|(|\lambda|-1)^{1-\alpha}}
		\]
		for all $\lambda \in \mathbb{C}$ with $1 < |\lambda| < 2$. \vspace{2pt}
		\item $\sigma(T) \subset \overline{S^{1/\alpha}_{c}}$
		for some constant $c \geq 2$, where 
		$S^{1/\alpha}_{c}$ is the $(1/\alpha)$-Stolz domain defined by \eqref{eq:Stolz_domain}.
	\end{enumerate}
\end{proposition}

\begin{proof}
	The implications (i) $\Rightarrow$ (ii) 
	and (ii) $\Rightarrow$ (iii) follow from Lemma~\ref{lem:Banach_RK_cond}.
	It remains to prove the implication
	(iii) $\Rightarrow$ (i).
	
	Let $n \geq n_0$. By assumption, 
	$|1-\lambda| \leq c^{\alpha} (1 - |\lambda|)^{\alpha}$ 
	for all $\lambda \in \sigma(T)$. Hence,
	\begin{equation}
		\label{eq:quasi_bound}
		\|T^n (I-T)\| \leq c^{\alpha}\,C \max_{\lambda \in \overline{\mathbb{D}}}
		|\lambda|^n (1 - |\lambda|)^{\alpha}.
	\end{equation}
	Define $f(s) \coloneqq
	s^n(1-s)^{\alpha}$ for $0\leq s \leq 1$. Then 
	\[
	f'(s) = s^{n-1}(1-s)^{\alpha-1} ( n - (n+\alpha)s )
	\]
	for all $s \in [0,1)$. Therefore,
	\begin{equation}
		\label{eq:f_bound_quasi}
		\max_{0 \leq s \leq 1} f(s) = f\left(
		\frac{n}{n+\alpha}
		\right) =
		\left(
		\frac{n}{n+\alpha}
		\right)^n \left(
		\frac{\alpha}{n+\alpha}
		\right)^{\alpha}
	\end{equation}
	From \eqref{eq:quasi_bound} and \eqref{eq:f_bound_quasi},
	we conclude that statement (i) holds.
\end{proof}

The estimate \eqref{eq:spectrum_bound} is satisfied for
operators similar to quasi-multiplication operators studied in
\cite[Section~5]{Ng2020}.
Typical examples of quasi-multiplication operators 
are normal operators on Hilbert spaces and 
multiplication operators on $L^p$-spaces with 
$1 \leq p < \infty$. An analytic Toeplitz operator
on $H^2(\mathbb{T})$, which 
is not normal in general, is also a quasi-multiplication operator,
as shown in \cite[Example~5.1]{Ng2020}.

\section{Integral estimate for resolvents}
\label{sec:integral_estimate}
An integral estimate for resolvents that characterizes
bounded $C_0$-semigroups on Hilbert spaces was given in
\cite{Gomilko1999} and \cite{Shi2000}.
This characterization 
has the following discrete analogue; see \cite{Cojuhari2008} and \cite[Theorem~II.1.12]{Eisner2010} for the proof.
\begin{theorem}
	\label{thm:GSF}
	Let $X$ be a Hilbert space, and let $T \in \mathcal{L}(X)$ satisfy $r(T) \leq 1$.
	Then the following statements are equivalent:
	\begin{enumerate}[label=\upshape(\roman*), leftmargin=*, widest=ii]
		\item $T$ is power-bounded.
		\item There exists a constant $M>0$ such that 
		\[
		\sup_{1 < r < 2} (r-1)
		\int_0^{2\pi} \left(
		\|R(re^{i \theta}, T) x\|^2 + \|R(re^{i \theta}, T^*) x\|^2
		\right) d\theta \leq M \|x\|^2
		\]
		for all $x \in X$.
	\end{enumerate}
\end{theorem}

The following theorem gives 
an integral estimate to characterize
rates of decay for discrete operator semigroups
on Hilbert spaces.
This result is a generalization of \cite[Proposition~4.3]{Wakaiki2023ESAIM}, 
where the case $f(n) = n^{\alpha}$ with $0 <\alpha
< 1/2$ was considered for the rate of decay of an individual orbit.
\begin{theorem}
	\label{thm:int_cond_chara}
	Let $X$ be a Hilbert space, and let $Y$ be a Banach space.
	Let
	$T \in \mathcal{L}(X)$ be power-bounded, and let 
	$S \in \mathcal{L}(Y,X)$.
	Suppose that $\alpha > 0$ and $k \in \mathbb{N}$ satisfy $k > \alpha + 1/2$.
	If 
	$f \colon (0,\infty) \to (0,\infty)$ is $\alpha$-bounded
	regularly varying,
	then the following statements are equivalent: \vspace{3pt}
	\begin{enumerate}[label=\upshape(\roman*), leftmargin=*, widest=ii]
		\item $\displaystyle
		\|T^nS\| = O \left(
		\frac{1}{
			f(n)}
		\right)$ as $n \to \infty$. \vspace{3pt}
		\item There exists a constant $M>0$ such that 
		\[
		\sup_{1<r<2} F_k(r^2-1) \int_{0}^{2\pi}
		\|R(r e^{i \theta}, T)^k S y\|^2 d\theta \leq M \|y\|^2
		\]
		for all $y \in Y$,
		where 
		\begin{equation}
			\label{eq:Fk_def}
			F_k(s) \coloneqq s^{2k-1} f(1/s
			)^{2},\quad s >0.
		\end{equation}
	\end{enumerate}
\end{theorem}

The transference from 
the semigroup estimate to the integral estimate in
Theorem~\ref{thm:int_cond_chara} is obtained from
the following lemma.
\begin{lemma}
	\label{lem:decay_to_int_cond}
	Let $X$ and $Y$ be Banach spaces, and let $Z$ be a Hilbert space.
	Let $T \in \mathcal{L}(X)$ satisfy $r(T) \leq 1$, and let 
	$S_1 \in \mathcal{L}(X,Z)$ and 
	$S_2 \in \mathcal{L}(Y,X)$.
	Suppose that $\alpha \geq 0$ and $k \in \mathbb{N}$ satisfy $k > \alpha + 1/2$, and
	let 
	$f \colon (0,\infty) \to (0,\infty)$ be $\alpha$-bounded
	regularly varying.
	If $\|S_1 T^n S_2\| = O(1/f(n))$ as $n \to \infty$, then
	there exist a constant $M>0$ such that
	\[
	\sup_{1<r<2} F_k(r^2-1) \int_{0}^{2\pi}
	\|S_1 R(r e^{i \theta}, T)^k S_2 y\|^2 d\theta \leq M \|y\|^2
	\]
	for all $y \in Y$,
	where $F_k$ is as in Theorem~\ref{thm:int_cond_chara}.
\end{lemma}
\begin{proof}
	Parseval's equality yields
	\begin{align}
		\label{eq:Parseval_to_cn}
		\frac{1}{2\pi}
		\int_0^{2\pi} \|S_1 R(re^{i\theta},T)^k S_2y \|^2 d\theta 
		&=
		\sum_{n=0}^{\infty} 
		\left\|
		\binom{n+k-1}{k-1}
		\frac{S_1T^nS_2y}{r^{n+k}} \right\|^2 \\
		&\leq \sum_{n=0}^{\infty} \frac{c(n)}{r^{2n}} \|y\|^2
		\notag 
	\end{align}
	for all $r > 1$ and $y \in Y$,
	where
	\[
	c(n) \coloneqq 
	\left( 
	\binom{n+k-1}{k-1}\, \| S_1T^nS_2 \| \right)^2
	\]
	for $n \in \mathbb{N}_0$.
	Define
	$h(t) \coloneqq f(t)^2$ for $t >0$.
	Since $\| S_1T^nS_2 \|^2 = O(1/h(n))$ as $n \to \infty$ by assumption,
	we have that 
	\[
	c(n) = O \left(
	\frac{n^{2(k-1)}}{h(n)}
	\right)
	\] 
	as $n \to \infty$.
	Lemma~\ref{lem:basic_property}.b) shows that
	$h$ is $(2\alpha)$-bounded regularly varying.
	By the estimate \eqref{eq:Parseval_to_cn} and
	Proposition~\ref{prop:cn_bound} with 
	$\beta = 2(k-1) > 2\alpha-1$, 
	there exist $M>0$ and $r_1 >1$ such that 
	\[
	\int_0^{2\pi} \|S_1 R(re^{i\theta},T)^k S_2y \|^2 d\theta \leq 
	\frac{M \|y\|^2}{F_k(r^2-1)}
	\]
	for all $r \in (1,r_1)$ and $y \in Y$.
	Thus, the desired estimate holds.
\end{proof}

Next, we show that the integral estimate leads to the semigroup estimate.
\begin{lemma}
	\label{lem:int_cond_to_decay}
	Let $X$ be a Hilbert space, and let $Y$ be a Banach space.
	Let
	$T \in \mathcal{L}(X)$ be power-bounded, and let $S \in \mathcal{L}(Y,X)$.
	If there exist a function $F:(0,\infty) \to (0,\infty)$
	and constants $M>0$ and $k \in \mathbb{N}$ such that
	\begin{equation}
		\label{eq:int_estimate_k}
		\sup_{1<r<2} F(r^2-1) \int_{0}^{2\pi}
		\|R(r e^{i \theta}, T)^k S y\|^2 d\theta \leq M \|y\|^2
	\end{equation}
	for all $y \in Y$,
	then
	\begin{equation}
		\label{eq:T_power_from_int}
		\|T^n S\| = O \left(
		\frac{1}{n^{k-1/2}F(1/n)^{1/2}}
		\right)
	\end{equation}
	as $n \to \infty$.
\end{lemma}
\begin{proof}
	Let 
	$n \in \mathbb{N}$,
	$x \in X$, and $y \in Y$. 
	Define 
	\[
	p_k(n) \coloneqq 
	\binom{n+k}{k}.
	\]
	By Lemma~\ref{lem:T_power_rep},
	\begin{align*}
		\langle T^n Sy, x \rangle &=
		\left\langle 
		\frac{r^{n+k+1}}{2\pi p_k(n)} \int_0^{2\pi} e^{i \theta (n+k+1)}
		R(r e^{i \theta},T)^{k+1}Sy d\theta, x
		\right\rangle \\
		&=
		\frac{r^{n+k+1}}{2\pi p_k(n)}
		\left\langle 
		\int_0^{2\pi} e^{i \theta (n+k+1)}
		R(r e^{i \theta},T)^kSy d\theta, R(r e^{-i \theta},T^*)x
		\right\rangle.
	\end{align*}
	Theorem~\ref{thm:GSF}
	shows that 
	\begin{equation}
		\label{eq:T_adj_estimate}
		\sup_{1 < r< 2}
		(r-1)
		\int_0^{2\pi}
		\|R(r e^{-i \theta},T^*)x\|^2 d\theta \leq M_1\|x\|^2
	\end{equation}
	for all $x \in X$ and some constant $M_1 >0$.
	Combining 
	the Cauchy--Schwarz inequality with
	the estimates \eqref{eq:int_estimate_k} and \eqref{eq:T_adj_estimate}, we obtain
	\begin{align*}
		|\langle T^n Sy, x \rangle|  &\leq \frac{r^{n+k+1}}{2\pi p_k(n)}
		\left(
		\int_0^{2\pi}
		\|R(r e^{i \theta},T)^kSy\|^2 d\theta
		\right)^{1/2}
		\left(
		\int_0^{2\pi}
		\|R(r e^{-i \theta},T^*)x\|^2 d\theta
		\right)^{1/2} \\
		&\leq 
		\frac{r^{n+k+1}}{2\pi p_k(n) }
		\sqrt{\frac{MM_1}{(r-1)F(r^2-1)}}
		\|y\|\, \|x\|
	\end{align*}
	for all $r \in (1,2)$.
	Set
	\[
	r = \left(1 + \frac{1}{n} \right)^{1/2}.
	\]
	Then
	\[
	r-1  \geq \frac{1}{3n}.
	\]
	Since $p_k(n) \geq n^k/k!$,
	we obtain the desired estimate \eqref{eq:T_power_from_int}.
\end{proof}

Combining Lemmas~\ref{lem:decay_to_int_cond} and \ref{lem:int_cond_to_decay},
we prove Theorem~\ref{thm:int_cond_chara}.
\begin{proof}[Proof of Theorem~\ref{thm:int_cond_chara}.]
	The implication (i) $\Rightarrow $ (ii) follows from
	Lemma~\ref{lem:decay_to_int_cond}.
	Since the function $F_k$ defined by \eqref{eq:Fk_def}
	satisfies
	\[
	\frac{1}{n^{k-1/2}F_k(1/n)^{1/2}} =
	\frac{1}{f(n)}
	\]
	for $n \in \mathbb{N}$, 
	Lemma~\ref{lem:int_cond_to_decay} with 
	$F=F_k$ shows that the implication (ii) $\Rightarrow $ (i) holds.
\end{proof}

We conclude this section by
presenting a characterization of  Ritt operators
on Hilbert spaces
as a corollary of 
Theorem~\ref{thm:int_cond_chara},
\begin{corollary}
	\label{coro:Ritt2}
	Let $X$ be a Hilbert space. Let $T \in \mathcal{L}(X)$
	and $k \in \mathbb{N}$.
	Then the following statements are equivalent:
	\begin{enumerate}[label=\upshape(\roman*), leftmargin=*, widest=ii]
		\item $T$ is a Ritt operator.
		\item $T$ is power-bounded, and there exists 
		a constant $M>0$ such that 
		\[
		\sup_{1<r<2}
		(r-1)^{2k-1}
		\int_0^{2\pi}
		\|
		(re^{i\theta}-1)R(re^{i \theta},T)^{k+1}x
		\|^2 d\theta
		\leq M \|x\|^2
		\]
		for all $x \in X$.
	\end{enumerate}
\end{corollary}
\begin{proof}
	By Theorems~\ref{thm:Ritt_characterization}
	and \ref{thm:int_cond_chara},
	it suffices to prove that the following statements
	are equivalent under the assumption that 
	$T$ is power-bounded:
	\begin{enumerate}[label=\upshape(ii-\alph*), leftmargin=*, widest=b]
		\item There exists $M_1>0$ such that 
		\begin{equation}
			\label{eq:resolvent_ritt_power1}
			\sup_{1<r<2}
			(r-1)^{2k-1}
			\int_0^{2\pi}
			\|
			(re^{i\theta}-1)R(re^{i \theta},T)^{k+1}x
			\|^2 d\theta
			\leq M_1 \|x\|^2
		\end{equation}
		for all $x \in X$.
		\item There exists $M_2>0$ such that 
		\begin{equation}
			\label{eq:resolvent_ritt_power2}
			\sup_{1<r<2}
			(r-1)^{2k-1}
			\int_0^{2\pi}
			\|
			R(re^{i \theta},T)^{k+1} (I-T)x
			\|^2 d\theta
			\leq M_2 \|x\|^2
		\end{equation}
		for all $x \in X$.
	\end{enumerate}
	
	Since $(a+b)^2 \leq 2 (a^2+b^2)$ for all $a,b \geq 0$, 
	we deduce from \eqref{eq:resolvent_change} that
	\[
	\|
	R(\lambda,T)^{k+1} (I-T)x
	\|^2 \leq 
	2\|
	(\lambda-1) R(\lambda,T)^{k+1}x
	\|^2  + 2\|
	R(\lambda,T)^kx
	\|^2
	\]
	and
	\[\|
	(\lambda-1) R(\lambda,T)^{k+1}x
	\|^2  \leq 
	2	\|
	R(\lambda,T)^{k+1} (I-T)x
	\|^2 + 2\|
	R(\lambda,T)^kx
	\|^2
	\]
	for all $\lambda \in \mathbb{E}$ and $x \in X$.
	Moreover, since $T$ is a Kreiss operator, it follows that 
	there exists $M_3 >0$ such that 
	\[
	\|
	R(\lambda,T)^k x
	\| \leq \frac{M_3\|R(\lambda,T)x\|}{(|\lambda| - 1)^{k-1}}
	\]
	for all $\lambda \in \mathbb{E}$ and $x \in X$.
	Combining this with Theorem~\ref{thm:GSF}, we see that 
	\[
	\sup_{1<r<2}
	(r-1)^{2k-1}
	\int_0^{2\pi}
	\|
	R(re^{i \theta},T)^kx
	\|^2 d\theta
	\leq M_4 \|x\|^2
	\]
	for all $x \in X$ and some constant $M_4>0$.
	Thus, statements~(ii-a) and (ii-b) are equivalent. 
\end{proof}

\section{Decay rates of perturbed discrete semigroups}
\label{sec:perturbation}
We use the following Sherman--Morrison--Woodbury formula
as in the robustness analysis for strong stability
of discrete operator semigroups established in \cite{Paunonen2015}. This formula
can be obtained from a direct calculation, and hence the proof is omitted.
\begin{lemma}
	Let $X$ and $Y$ be Banach spaces.
	Let $A \in \mathcal{L}(X)$, $B \in \mathcal{L}(Y,X)$, $C \in \mathcal{L}(X,Y)$, 
	and $\lambda \in \varrho(A)$.
	If $1 \in \varrho(C R(\lambda,A)B)$, then $\lambda \in \varrho(A+BC)$
	and 
	\begin{equation}
		\label{eq:SMW}
		R(\lambda,A+BC) = 
		R(\lambda,A) + R(\lambda,A) B(I - CR(\lambda,A)B)^{-1} C R(\lambda,A).
	\end{equation}
\end{lemma}

The following corollary of 
Theorem~\ref{thm:GSF} provides a condition for
the preservation of
boundedness and decay properties under perturbations.
Similar conditions were obtained for polynomial
stability
and strong stability of $C_0$-semigroups in
\cite[Theorems~4.1 and 4.2]{Rastogi2020}.
\begin{corollary}
	\label{coro:robustness}
	Let $X$ be a Hilbert space, and 
	let $Y$ be a Banach space.
	Let $T \in \mathcal{L}(X)$ be power-bounded, and let
	$S \in \mathcal{L}(Y,X)$. 
	If $D \in \mathcal{L}(X)$ commutes with $T$ and 
	satisfies 
	$\sup_{\lambda \in \mathbb{E}}\|R(\lambda,T)D\| < 1$,
	then the following statements hold:
	\begin{enumerate}[label=\upshape\alph*), leftmargin=*, widest=b]
		\item $T+D$ is power-bounded.
		\item 
		Let $f \colon (0,\infty) \to (0,\infty)$ 
		be $\alpha$-bounded
		regularly varying for some $\alpha > 0$.
		If
		$\|T^nS\| =  O(1/f(n))$ as $n \to \infty$,
		then 
		$\|(T+D)^nS\| =  O(1/f(n))$ as $n \to \infty$.
	\end{enumerate}
\end{corollary}
\begin{proof}
	Define $\delta \coloneqq \sup_{z \in \mathbb{E}}\|R(z,T)D\| < 1$,
	and let $\lambda \in \mathbb{E}$.
	Since 
	$1 \in \varrho(R(\lambda,T)D)$,
	the Sherman--Morrison--Woodbury formula \eqref{eq:SMW} gives
	\begin{equation}
		\label{eq:resol_TD_rep}
		R(\lambda,T+D) = R(\lambda,T) + R(\lambda,T)D (I - R(\lambda,T)D)^{-1} R(\lambda,T).
	\end{equation}
	Moreover, 
	\begin{equation}
		\label{eq:D_bound}
		\|R(\lambda,T)D\| \,\|(I - R(\lambda,T)D)^{-1}\| \leq \frac{\delta }{1-
			\delta}.
	\end{equation}
	a) By \eqref{eq:resol_TD_rep} and \eqref{eq:D_bound},
	\begin{equation}
		\label{eq:resol_TD_boundx}
		\|R(\lambda,T+D) x\| \leq  
		\frac{1}{1-\delta}\|R(\lambda,T)x\|
	\end{equation}
	for all $x \in X$.
	Since  $D$ commutes with $R(z,T)$ for all $z \in \mathbb{E}$,
	we have
	\[
	\sup_{z \in \mathbb{E}}\|R(z,T^*)D^*\| =
	\sup_{z \in \mathbb{E}}\|DR(z,T)\| = \delta.
	\]
	As in \eqref{eq:resol_TD_boundx}, it follows that
	\begin{equation}
		\label{eq:resol_TD_bound_adj}
		\|R(\lambda,T^*+D^*) x\| \leq  
		\frac{1}{1-\delta}\|R(\lambda,T^*)x\|
	\end{equation}
	for all $x \in X$.
	By Theorem~\ref{thm:GSF}, there exists $M_1>0$ such that 
	\[
	\sup_{1 < r < 2} (r-1)
	\int_0^{2\pi} \left(
	\|R(re^{i \theta}, T) x\|^2 + \|R(re^{i \theta}, T^*) x\|^2
	\right) d\theta \leq M_1 \|x\|^2
	\]
	for all $x \in X$. This, together with the estimates \eqref{eq:resol_TD_boundx}
	and \eqref{eq:resol_TD_bound_adj}, yields
	\begin{align*}
		&\sup_{1 < r < 2} (r-1)
		\int_0^{2\pi} \left(
		\|R(re^{i \theta}, T+D) x\|^2 + \|R(re^{i \theta}, T^*+D^*) x\|^2
		\right) d\theta \leq \frac{M_1}{(1-\delta)^2} \|x\|^2
	\end{align*}
	for all $x \in X$.
	Using Theorem~\ref{thm:GSF} again, we have that 
	$T+D$ is power-bounded.
	
	b)
	From the commutativity of $D$ and $T$, we see that 
	$R(\lambda,T)$ commutes with 
	\[
	H(\lambda)\coloneqq 
	R(\lambda,T)D(I-R(\lambda,T)D)^{-1}.
	\] 
	Let $k \in \mathbb{N}$ satisfy $k >  \alpha + 1/2$.
	Since \eqref{eq:resol_TD_rep} yields
	\[
	R(\lambda,T+D)^k = (1+H(\lambda))^k R(\lambda,T)^k,
	\]
	we have by \eqref{eq:D_bound},
	\begin{equation}
		\label{eq:power_resol_bound}
		\|R(\lambda,T+D)^k Sy\| \leq 
		\frac{1}{( 1-\delta)^k}
		\|R(\lambda,T)^k Sy\|
	\end{equation} 
	for all $y \in Y$.
	By~Theorem~\ref{thm:int_cond_chara}, 
	there exists $M_2>0$ such that
	\[
	\sup_{1<r<2} F_k(r^2-1) \int_{0}^{2\pi}
	\|R(r e^{i \theta}, T)^k S y\|^2 d\theta \leq M_2 \|y\|^2
	\]
	for all $y \in Y$,
	where $F_k$ is as in that theorem.
	Combining this with \eqref{eq:power_resol_bound},
	we obtain
	\[
	\sup_{1<r<2} F_k(r^2-1) \int_{0}^{2\pi}
	\|R(r e^{i \theta}, T+D)^k S y\|^2 d\theta \leq \frac{M_2}{(1-\delta)^{2k}} \|y\|^2
	\]
	for all $y \in Y$,
	Thus, Theorem~\ref{thm:int_cond_chara} 
	shows that $\|(T+D)^nS\| = O(1/f(n))$ as $n \to \infty$.
\end{proof}

\section{Relation between 
	semigroup decay and summability conditions}
\label{sec:summability}
Now we study the relation between the rate of decay 
of discrete operator semigroups and summability conditions.
\begin{proposition}
	\label{prop:weighted_sum_decay}
	Let $X$ and $Y$ be Banach spaces. Let $T \in \mathcal{L}(X)$,
	$S \in \mathcal{L}(Y,X)$, and $p > 0$.
	Then the following statements hold:
	\begin{enumerate}[label=\upshape\alph*), leftmargin=*, widest=b]
		\item 
		Assume that $T$ is power-bounded.
		If there exist functions $f,G\colon \mathbb{N} \to (0,\infty)$ 	such that
		\begin{equation}
			\label{eq:f_sum}
			\sum_{m=1}^{n} f(m) \|T^mSy\|^p \leq G(n) \|y\|^p
		\end{equation}
		for all $n \in \mathbb{N}$ and $y \in Y$,
		then 
		\begin{equation}
			\label{eq:T_decay}
			\|T^nS\| \leq  \frac{KG(n)^{1/p}}{F(n)^{1/p}},
		\end{equation}
		for all $n \in \mathbb{N}$,
		where $K \coloneqq \sup_{n \in \mathbb{N}_0} \|T^n\|$ and
		\begin{equation}
			\label{eq:summarability_Fn}
			F(n) \coloneqq \sum_{m=1}^n f(m).
		\end{equation}
		\item 
		Let $f,g\colon (0,\infty) \to (0,\infty)$ be $\alpha$-bounded
		regularly varying for some constant $\alpha \geq 0$.
		Define 
		\[
		F(n) \coloneqq \sum_{m=1}^n f(m)g(m)\quad \text{and} \quad 
		G(n) \coloneqq \sum_{m=1}^{n} \frac{1}{mg(m)}
		\]
		for $n \in \mathbb{N}$.
		If $\|T^nS\| = O(1/F(n)^{1/p})$ as $n \to \infty$, then
		there exists a constant $C>0$ such that 
		\[
		\sum_{m=1}^n f(m) \|T^m Sy\|^p \leq 
		CG(n) \|y\|^p
		\]
		for all $n \in \mathbb{N}$ and $y \in Y$.
	\end{enumerate}
\end{proposition}
\begin{proof}
	a)
	For all $n \in \mathbb{N}$ and $y \in Y$,
	\begin{align*}
		F(n) \|T^nSy\|^p \leq \sum_{m=1}^n f(m)
		\|T^{n-m}\|^p\, \|T^m Sy\|^p \leq 
		K^p \sum_{m=1}^n f(m)\|T^m Sy\|^p.
	\end{align*}
	From this and \eqref{eq:f_sum},
	we obtain the desired estimate \eqref{eq:T_decay}.
	
	b) By assumption,
	there exists $t_0 >0$ such that 
	\begin{equation}
		\label{eq:fg_estimates}
		\frac{tf'(t)}{f(t)} \leq  \alpha 
		\quad \text{and} \quad 
		\frac{tg'(t)}{g(t)} \leq \alpha 
	\end{equation}
	for a.e.~$t \geq t_0$.
	For all $t_1 \geq t_0$,
	integration by parts yields
	\begin{align*}
		\int_{t_0}^{t_1} f(t)g(t)dt &= 
		t_1f(t_1)g(t_1) - t_0f(t_0)g(t_0) -
		\int_{t_0}^{t_1} t (f'(t)g(t) + f(t)g'(t)) dt \\
		&\geq 
		t_1f(t_1)g(t_1) - t_0f(t_0)g(t_0)- 2 \alpha 
		\int_{t_0}^{t_1} f(t)g(t) dt.
	\end{align*}
	If $t_1 \geq 2t_0$, then
	\[
	t_1f(t_1)g(t_1) \geq  2t_0f(t_0)g(t_0),
	\]
	and hence
	\begin{equation}
		\label{eq:int_fg_estimate}
		\int_{t_0}^{t_1} f(t)g(t)dt \geq 
		\frac{t_1f(t_1)g(t_1)}{2(1+2\alpha)}.
	\end{equation}

	By assumption,
	there exist $M_0 >0$ and $n_0 \in \mathbb{N}$
	such that for all $n \geq n_0$,
	\begin{equation}
		\label{eq:TS_Fg_bound}
		\|T^nS\| \leq \frac{M_0}{F(n)^{1/p}}.
	\end{equation}
	Let $n_1 \in \mathbb{N}$ satisfy
	$ n_1 \geq  \max \{n_0,\,2t_0 \}$.
	By \eqref{eq:int_fg_estimate}, we have
	\[
	F(m) \geq 
	\int_0^{t_0} f(t)g(t) dt + 
	\int_{t_0}^{m} f(t)g(t) dt
	\geq \frac{mf(m)g(m)}{2(1+2\alpha)}
	\]
	for all $m \geq n_1$.
	This gives
	\begin{equation}
		\label{eq:f_Fg_bound}
		\sum_{m=n_1+1}^{n}   \frac{f(m)}{F(m)} \leq 
		2(1+2\alpha) \sum_{m=n_1+1}^{n} \frac{1}{mg(m)} \leq 
		2(1+2\alpha) G(n)
	\end{equation}
	for all $n \geq n_1+1$.
	Using 
	the estimates~\eqref{eq:TS_Fg_bound} and \eqref{eq:f_Fg_bound}, we obtain
	\begin{align*}
		\sum_{m=1}^n f(m) \|T^m Sy\|^p &\leq 
		\left(
		\sum_{m=1}^{n_1} f(m) \|T^m S \|^p + M_0^p
		\sum_{m=n_1+1}^{n}   \frac{f(m)}{F(m)}
		\right) \|y\|^p \\
		&\leq 
		\left(
		\sum_{m=1}^{n_1} f(m) \|T^m S \|^p + 2(1+2\alpha)M_0^p G(n)
		\right) \|y\|^p
	\end{align*}
	for all $n \geq n_1+1$ and $y\in Y$.
	Thus, the desired conclusion holds.
\end{proof}

The following result is a special case of Proposition~\ref{prop:weighted_sum_decay}.
\begin{corollary}
	\label{coro:summarability}
	Let $X$ and $Y$ be Banach spaces. 
	Let $T \in \mathcal{L}(X)$,  
	$S \in \mathcal{L}(Y,X)$, and $p>0$.
	Let 
	$f\colon (0,\infty) \to (0,\infty)$ be $\alpha$-bounded
	regularly varying for some  $\alpha \geq 0$, and define 
	$F$ by \eqref{eq:summarability_Fn}.
	Then the following statements hold:
	\begin{enumerate}[label=\upshape\alph*), leftmargin=*, widest=b]
		\item 
		Assume that $T$ is power-bounded.
		If there exists a constant $C >0$ such that 
		\begin{equation}
			\label{eq:f_sum_coro1}
			\sum_{n=1}^{\infty} f(n) \|T^nSy\|^p \leq C \|y\|^p
		\end{equation}
		for all $y \in Y$, then $\|T^nS\| = O(1/F(n)^{1/p})$ as $n \to \infty$.
		\item 
		If $\|T^nS\| = O(1/F(n)^{1/p})$ as $n \to \infty$, then
		there exists a constant $C >0$ such that 
		\begin{equation}
			\label{eq:f_sum_coro2}
			\sum_{m=1}^{n} f(m) \|T^mSy\|^p \leq C \log n \|y\|^p
		\end{equation}
		for all $n \in \mathbb{N}$ and $y \in Y$.
	\end{enumerate}
\end{corollary}
\begin{proof}
	Statement~a) follows from Proposition~\ref{prop:weighted_sum_decay}.a) with $G(n) \equiv C$.
	If $g(t) \equiv 1$ in 
	Proposition~\ref{prop:weighted_sum_decay}.b), then
	\[
	G(n) = \sum_{m=1}^n \frac{1}{m} \leq 2 \log (n+1)
	\]  
	for all $n \in \mathbb{N}$.
	Hence, 
	Proposition~\ref{prop:weighted_sum_decay}.b) shows that 
	statement~b) is true.
\end{proof}
There is a logarithmic gap 
in \eqref{eq:f_sum_coro1} and \eqref{eq:f_sum_coro2}.
However, in the special case when $S = I - T$ and $f(n) = n^{\alpha p - 1}$,
this gap can be bridged for a multiplication operator  $T$ on $L^q$-space 
with $1\leq q \leq p < \infty$. 
An analogous result for $C_0$-semigroups was obtained under the condition $\alpha p =1$  in
\cite[Theorem~3.8]{Wakaiki2024JMAA}.
For discrete semigroups, we use the spectral property based on
$\delta$-Stolz domains.
\begin{proposition}
	\label{prop:multiplication_summability}
	Let 
	$(\Omega, \mu)$ be a $\sigma$-finite measure space.
	Let $\phi \colon \Omega \to \mathbb{C}$ 
	be measurable with essential range $\phi_{\mathrm{ess}} (\Omega)$
	in $\overline{\mathbb{D}}$.
	Let $1 \leq q \leq p < \infty$, and let 
	$T$ be the multiplication operator induced by $\phi$
	on $L^q(\Omega,\mu)$, i.e., $Tx = \phi x$
	for $x \in L^q(\Omega,\mu)$.
	Then the following statements are equivalent for a
	fixed $\alpha \in (0,1)$:
	\begin{enumerate}[label=\upshape(\roman*), leftmargin=*, widest=ii]
		\item $\|T^n(I-T)\| = O(n^{-\alpha} )$ as $n \to \infty$.
		\item There exists a constant $C>0$ such that 
		\[
		\sum_{n=1}^{\infty} n^{\alpha p-1} \|T^n (I-T)x\|^p \leq C \|x\|^p
		\]
		for all $x \in L^q(\Omega,\mu)$.
	\end{enumerate}
\end{proposition}
\begin{proof}
	Since
	$\phi(\omega) \in \phi_{\mathrm{ess}}(\Omega)$
	for $\mu$-a.e.~$\omega \in \Omega$ (see, e.g., 
	\cite[Chapter VII, Exercise~19]{Lang1993}), 
	$T$ is power-bounded.
	The implication (ii) $\Rightarrow$ (i) is obtained from
	Proposition~\ref{prop:weighted_sum_decay}.a) with 
	$f(n) =
	n^{\alpha p - 1}$, $G(n) \equiv C$, and $S = I-T$.
	It remains to prove the implication
	(i) $\Rightarrow$ (ii).

	Recall from \cite[Proposition~I.4.10]{Engel2000} that
	$\sigma(T) = \phi_{\mathrm{ess}}(\Omega)$.
	Since 
	$\sigma(T) \subset \overline{S^{1/\alpha}_{c}}$
	by Lemma~\ref{lem:Banach_RK_cond},
	we obtain
	\begin{equation}
		\label{eq:phi_omega_ae}
		\phi(\omega) \in \overline{S^{1/\alpha}_{c}}\quad 
		\text{for $\mu$-a.e.~$\omega \in \Omega$}.
	\end{equation}
	
	Define $\beta \coloneqq 
	\alpha p - 1$.
	For all $x \in L^q(\Omega, \mu)\setminus \{ 0\}$,
	\begin{align}
		\label{eq:Lp_sum}
		\frac{1}{\|x\|^p}
		\sum_{n=1}^{\infty} n^{\beta }\|T^n(I-T)x\|^p  =
		\sum_{n=1}^{\infty}n^{\beta }
		\left(
		\int_{\Omega} |\phi(\omega)^n (1 - \phi(\omega) ) |^q
		\frac{|x(\omega)|^q}{\|x\|^q}
		d\mu(\omega)
		\right)^{p/q} 
	\end{align}
	Since $p/q \geq 1$,
	Jensen's inequality yields
	\begin{align}
		\label{eq:Lp_sum_Jensen}
		\left(
		\int_{\Omega} 
		|\phi(\omega)^n (1 - \phi(\omega) ) |^q
		\frac{|x(\omega)|^q}{\|x\|^q}
		d\mu(\omega)
		\right)^{p/q}\leq 
		\int_{\Omega} |\phi(\omega)|^{np} 
		|1-\phi(\omega)|^p 
		\frac{|x(\omega)|^q}{\|x\|^q} d\mu(\omega)
	\end{align}
	for all $n \in \mathbb{N}$ and $x \in L^q(\Omega, \mu)
	\setminus \{ 0\}$.
	From
	\eqref{eq:phi_omega_ae}--\eqref{eq:Lp_sum_Jensen},
	we now investigate
	\[
	|1 - \lambda|^p
	\sum_{n=1}^{\infty} n^{\beta} |\lambda|^{np}
	\]
	for $\lambda \in \overline{S^{1/\alpha}_c}\setminus \{1 \}$.
	Since $\overline{S^{1/\alpha}_c} \cap \mathbb{T} = \{ 1\}$,
	it follows that 
	$\lambda
	\in \overline{S^{1/\alpha}_c} \setminus \{1 \}$ satisfies
	$|\lambda|< 1 $.
	If $\lambda \in \overline{S^{1/\alpha}_{c}}$
	satisfies $|\lambda| \leq e^{-1}$, then
	\[
	|1-\lambda|^p\sum_{n=1}^{\infty} n^{\beta} |\lambda|^{np} \leq 
	(1+e^{-1})^{p}
	\sum_{n=1}^{\infty} n^{\beta} e^{-np} \eqqcolon C_1 < \infty.
	\]
	Suppose that $\lambda \in \overline{S^{1/\alpha}_{c}}$
	satisfies $e^{-1} < |\lambda| < 1$, and define 
	$r \coloneqq 1/|\lambda| $. If $n \leq t < n+1$
	for some $n \in \mathbb{N}$, then
	\[
	\frac{n^{\beta}}{r^{np}} \leq \frac{t^{\beta}}{r^{(t-1)p}} =
	t^{\beta} e^{-(p\log r) (t-1)}.
	\]
	Therefore,
	\begin{align*}
		\sum_{n=1}^{\infty} n^{\beta} |\lambda|^{np} 
		&=
		\sum_{n=1}^{\infty} \frac{n^{\beta}}{r^{np}} \\
		&\leq 
		\int_1^{\infty} t^{\beta} e^{-(p\log r) (t-1)} dt \\
		&\leq 
		\frac{e^p}{(p \log r)^{\beta+ 1}}	
		\int_0^{\infty} t^{\beta} e^{-t} dt \\
		&\leq
		\frac{e^p \Gamma(\beta + 1)}{(p \log r)^{\beta + 1}}.
	\end{align*}
	Since $r-1 \leq 2\log r$, we have 
	\[
	\sum_{n=1}^{\infty} n^{\beta} |\lambda|^{np} \leq 
	\frac{C_2}{ (r-1)^{\beta+1}},
	\quad \text{where~}
	C_2 \coloneqq e^p\Gamma(\beta + 1) 
	\left(
	\frac{ 2}{p }
	\right)^{\beta + 1}.
	\]
	From $\lambda \in \overline{S^{1/\alpha}_{c}}$, it follows that 
	\[
	|1 - \lambda|^{1/\alpha} \leq c (1 - |\lambda|).
	\]
	Recalling that $\beta+1 = \alpha p$,
	we obtain
	\[
	|1-\lambda|^p \sum_{n=1}^{\infty} n^{\beta} |\lambda|^{np} \leq  c^{\alpha p }(1 - |\lambda|)^{\alpha p}
	\frac{C_2|\lambda|^{\beta+1}}{(1 - |\lambda|)^{\beta+1}}
	\leq c^{\alpha p }C_2.
	\]
	
	We have shown that 
	\begin{equation}
		\label{eq:mo_sum_bound}
		|1-\lambda|^p \sum_{n=1}^{\infty} n^{\beta} |\lambda|^{np}
		\leq C_3 \coloneqq \max \{ C_1,\, c^{\alpha p }C_2\}
	\end{equation}
	for all $\lambda \in \overline{S^{1/\alpha}_{c}} \setminus \{1 \}$.
	Let $\Omega_0 \coloneqq \Omega \setminus \{\omega \in \Omega: \phi(\omega) = 1\}$. 
	Then
	\begin{align*}
		\sum_{n=1}^{\infty}n^{\beta }
		\int_{\Omega} |\phi(\omega)|^{np} 
		|1-\phi(\omega)|^p 
		\frac{|x(\omega)|^q}{\|x\|^q} d\mu(\omega)  =
		\sum_{n=1}^{\infty}n^{\beta }
		\int_{\Omega_0} |\phi(\omega)|^{np} 
		|1-\phi(\omega)|^p 
		\frac{|x(\omega)|^q}{\|x\|^q} d\mu(\omega)
	\end{align*}
	for all $x \in L^q(\Omega, \mu)\setminus \{ 0\}$.
	Using \eqref{eq:phi_omega_ae}, \eqref{eq:mo_sum_bound}, 
	and 
	the monotone convergence theorem, we obtain
	\begin{align*}
		\sum_{n=1}^{\infty}n^{\beta }
		\int_{\Omega_0} |\phi(\omega)|^{np} 
		|1-\phi(\omega)|^p 
		\frac{|x(\omega)|^q}{\|x\|^q} d\mu(\omega)
		&\leq \frac{C_3}{\|x\|^q} \int_{\Omega_0} |x(\omega)|^q d\mu(\omega) \leq C_3
	\end{align*}
	for all $x \in L^q(\Omega, \mu)\setminus \{ 0\}$.
	Thus, statement~(ii) with $C = C_3$ holds.
\end{proof}

We have seen the relation between semigroup decay and 
summability conditions.
On the other hand,
the relation between semigroup decay
and resolvent growth has been investigated 
in Section~\ref{sec:resolvent_growth}.
The following proposition shows how
summability conditions can be transferred to resolvent estimates.
\begin{proposition}
	\label{prop:weighted_sum_resolvent}
	Let $X$, $Y$, and $Z$ be Banach spaces. Let $T \in \mathcal{L}(X)$
	satisfy $r(T) \leq 1$, and 
	$S_1 \in \mathcal{L}(X,Z)$ and $S_2 \in \mathcal{L}(Y,X)$.
	Let 
	$f\colon (0,\infty) \to (0,\infty)$ be $\alpha$-bounded 
	regularly varying for some $\alpha \geq 0$, and assume that 
	there exist constants 
	$C>0$ and $p>1$ such that
	\begin{equation}
		\label{eq:f_sum_resol}
		\sum_{n=1}^{\infty} f(n) \|S_1T^nS_2y\|^p \leq C \|y\|^p
	\end{equation}
	for all $y \in Y$.
	Then for all $k \in \mathbb{N}$ with $k > (\alpha+1)/p$,
	there exists  a constant $M>0$ such that 
	\[
	\|S_1R(\lambda,T)^kS_2\| \leq \frac{M}{
		(|\lambda|^q - 1)^{k-1/p} f(1/(|\lambda|^q - 1))^{1/p}
	}
	\]
	for all $\lambda \in \mathbb{C}$ with $1 < |\lambda| < 2$, 
	where $q > 1$ satisfies $1/p+1/q = 1$. 
\end{proposition}
\begin{proof}
	Let $k \in \mathbb{N}$ satisfy $k > (\alpha+1)/p$.
	Using \eqref{eq:Rk_rep}, we obtain
	\[
	S_1R(\lambda,T)^{k}S_2 =
	\frac{S_1S_2}{\lambda^k} 
	+ 
	\frac{1}{\lambda^k} 
	\sum_{n=1}^{\infty} 
	\binom{n+k-1}{k-1}
	\frac{S_1T^n S_2}{\lambda^{n}} .
	\]
	H\"older's inequality yields
	\begin{equation}
		\label{eq:CS_weighted_case}
		\left\|\sum_{n=1}^{\infty} 
		\binom{n+k-1}{k-1}
		\frac{S_1T^n S_2y}{\lambda^{n}}\right\| \leq 
		\left(
		\sum_{n=1}^{\infty}
		f(n) \|S_1T^n S_2y \|^p
		\right)^{1/p}
		\left(
		\sum_{n=1}^{\infty}
		\frac{c(n)}{ |\lambda|^{nq}} 
		\right)^{1/q}
	\end{equation}
	for all $\lambda \in \mathbb{E}$, where 
	\[
	c(n) \coloneqq 	\frac{(n+k)^{(k-1)q}}{ f(n)^{q/p}},\quad n \in \mathbb{N}.
	\]
	
	By Lemma~\ref{lem:basic_property}.b), $h\coloneqq f^{q/p}$ is $(\alpha q/p)$-bounded regularly varying.
	Since $k>(\alpha+1)/p$ is equivalent to
	\[
	(k-1)q > \frac{\alpha q}{p}  - 1
	\]
	and since
	\[
	c(n) = O\left(
	\frac{n^{(k-1)q}}{h(n)}
	\right)
	\]
	as $n \to \infty$, Proposition~\ref{prop:cn_bound} shows that 
	\begin{equation}
		\label{eq:cn_sum_weighted_case}
		\sum_{n=1}^{\infty}
		\frac{c(n)}{r^{nq}} =
		O
		\left(
		\frac{1}{ (r^q - 1)^{(k-1)q + 1} h(1/(r^q-1))}
		\right)
	\end{equation}
	as $r \downarrow 1$. Applying
	\eqref{eq:f_sum_resol} and \eqref{eq:cn_sum_weighted_case}
	to \eqref{eq:CS_weighted_case}, we derive
	the desired estimate.
\end{proof}

\begin{remark}
	It is legitimate to ask if the converse of 
	Proposition~\ref{prop:weighted_sum_resolvent} is true.
	When $p = 2$ and $f(n) = n^{\alpha}$
	for $\alpha \geq 0$,
	the answer of 
	this question is positive for some classes of operators, as shown in
	literature on 
	admissible observation operators; see 
	\cite{Harper2006,Wynn2009,Wynn2009SCL,Wynn2010,
		Jacob2013,LeMerdy2014,Jacob2018admissibility}.
	The systematic study of admissible observation operators traces back to the work \cite{Weiss1989_observation} by Weiss for $C_0$-semigroups.
	We refer to 
	the survey article \cite{Jacob2004Survey} and the books \cite{Staffans2005,Tucsnak2009} for more information
	on admissibility for $C_0$-semigroups.
\end{remark}

\subsection*{Acknowledgements}
This work was supported in part by JSPS KAKENHI Grant Number
24K06866.
The
author would like to thank the editor and the anonymous referee for numerous helpful comments and
suggestions.

\end{document}